\documentclass[a4paper,11pt]{amsart}
\usepackage{amssymb,amscd,verbatim, amsthm,amsmath,amsgen,xspace,  mathrsfs}
\usepackage{latexsym}
\usepackage{amsfonts}
\usepackage{color}
\usepackage{ulem}
\usepackage[all]{xy}
\usepackage{etoolbox}

\def\dim{\mathop{\hbox {dim}}\nolimits}

\def\End{\mathop{\hbox {End}}\nolimits}

\newcommand{\pf}{\begin{proof}}
	\newcommand{\epf}{\end{proof}}
\newcommand{\eq}{\begin{equation}}
	\newcommand{\eeq}{\end{equation}}
\newcommand{\eqn}{\begin{equation*}}
	\newcommand{\eeqn}{\end{equation*}}

\newcommand{\frg}{\mathfrak{g}}
\newcommand{\frh}{\mathfrak{h}}

\newcommand{\frsl}{\mathfrak{sl}}

\newcommand{\tr}{\operatorname{tr}}

\newtheorem{thm}[equation]{Theorem}
\newtheorem{cor}[equation]{Corollary}
\newtheorem{lemma}[equation]{Lemma}
\newtheorem{prop}[equation]{Proposition}
\newtheorem{defi}[equation]{Definition}

\numberwithin{equation}{section}

\let\ssize\scriptstyle
\newif\ifFIRST\newdimen\MAXright\MAXright0pt
\def\sdynkin{\bgroup\eightpoint\dynkin}
\def\endsdynkin{\enddynkin\egroup}
\def\dynkin{\bgroup\FIRSTtrue\hskip.5em\setbox1\hbox{$\diagup$}%
	\setbox2\hbox{$\diagdown$}%
	\setbox0\hbox to2\wd1{\hrulefill}%
	\setbox3\hbox{$\bullet$}%
	\setbox4\hbox{$\times$}%
	\setbox7\hbox{$\circ$}
	\def\whiteroot##1{\ifFIRST\setbox5\hbox{$##1$}\ifdim\wd5>1.3em
		\hskip-.5em\hskip.5\wd5\fi\fi\FIRSTfalse
		\hskip-.25em\raise1.5\wd3\hbox to0pt{\hss\hskip.45em$
			\ssize##1$\hss}\copy7\hskip-.25em\setbox6\hbox{$##1$}
		\MAXright\wd6}
	\def\root##1{\ifFIRST\setbox5\hbox{$##1$}\ifdim\wd5>1.3em%
		\hskip-.5em\hskip.5\wd5\fi\fi\FIRSTfalse%
		\hskip-.25em\raise1.5\wd3\hbox to0pt{\hss\hskip.45em$%
			\ssize##1$\hss}\copy3\hskip-.25em\setbox6\hbox{$##1$}%
		\MAXright\wd6}%
	\def\whitedroot##1{\ifFIRST\setbox5\hbox{$##1$}\ifdim\wd5>1.3em
		\hskip-.5em\hskip.5\wd5\fi\fi\FIRSTfalse
		\hskip-.25em\lower1.8\wd3\hbox to0pt{\hss\hskip.45em$
			\ssize##1$\hss}\copy7\hskip-.25em\setbox6\hbox{$##1$}
		\MAXright\wd6}%
	\def\whiterroot##1{\hskip-.25em\copy7\hbox to0pt{\hskip.3em$\ssize##1$\hss}%
		\hskip-.25em\setbox6\hbox{\hskip.6em$##1##1$}%
		\MAXright\wd6}%
	\def\droot##1{\ifFIRST\setbox5\hbox{$##1$}\ifdim\wd5>1.3em%
		\hskip-.5em\hskip.5\wd5\fi\fi\FIRSTfalse%
		\hskip-.25em\lower1.8\wd3\hbox to0pt{\hss\hskip.45em$%
			\ssize##1$\hss}\copy3\hskip-.25em\setbox6\hbox{$##1$}%
		\MAXright\wd6}%
	\def\rroot##1{\hskip-.25em\copy3\hbox to0pt{\hskip.3em$\ssize##1$\hss}%
		\hskip-.25em\setbox6\hbox{\hskip.6em$##1##1$}%
		\MAXright\wd6}%
	\def\norroot##1{\hskip-.36em\copy4\hbox to0pt{\hskip.3em$\ssize##1$\hss}%
		\hskip-.48em\setbox6\hbox{\hskip.6em$##1##1$}%
		\MAXright\wd6}%
	\def\noroot##1{\ifFIRST\setbox5\hbox{$##1$}\ifdim\wd5>1.3em%
		\hskip-.5em\hskip.5\wd5\fi\fi\FIRSTfalse%
		\hskip-.36em\raise1.5\wd3\hbox to0pt{\hss\hskip.6em$%
			\ssize##1$\hss}\copy4\hskip-.38em\setbox6\hbox{$##1$}%
		\MAXright\wd6}%
	\def\nodroot##1{\ifFIRST\setbox5\hbox{$##1$}\ifdim\wd5>1.3em%
		\hskip-.5em\hskip.5\wd5\fi\fi\FIRSTfalse%
		\hskip-.36em\lower1.8\wd3\hbox to0pt{\hss\hskip.6em$%
			\ssize##1$\hss}\copy4\hskip-.38em\setbox6\hbox{$##1$}%
		\MAXright\wd6}%
	\def\nolink{\hskip\wd0}
	\def\link{\raise.22em\copy0}%
	\def\llink##1{\raise.32em\copy0\hskip-\wd0%
		\raise.12em\copy0\hskip-.5\wd0\hbox to0pt{\hss$##1$\hss}\hskip.5\wd0}%
	\def\lllink##1{\raise.22em\copy0\hskip-\wd0\raise.32em\copy0\hskip-\wd0%
		\raise.12em\copy0\hskip-.5\wd0\hbox to0pt{\hss$##1$\hss}\hskip.5\wd0}%
	\def\rootupright##1{\hbox to0pt{\raise.45em\copy1\hskip-.25em\raise1.3\ht1%
			\hbox{\copy3\hskip.3em$\ssize##1$}\hss}%
		\setbox6\hbox{\hskip.6em\copy1\copy1$##1##1$}%
		\ifdim\MAXright<\wd6\MAXright\wd6\fi}%
	\def\whiterootupright##1{\hbox to0pt{\raise.45em\copy1\hskip-.25em\raise1.3\ht1
			\hbox{\copy7\hskip.3em$\ssize##1$}\hss}
		\setbox6\hbox{\hskip.6em\copy1\copy1$##1##1$}
		\ifdim\MAXright<\wd6\MAXright\wd6\fi}
	\def\norootupright##1{\hbox to0pt{\raise.45em\copy1\hskip-.36em\raise1.3\ht1%
			\hbox{\copy4\hskip.3em$\ssize##1$}\hss}%
		\setbox6\hbox{\hskip.6em\copy1\copy1$##1##1$}%
		\ifdim\MAXright<\wd6\MAXright\wd6\fi}%
	\def\rootdownright##1{\hbox to0pt{\raise-.5em\copy2\hskip-.25em\raise-1.35\ht1%
			\hbox{\copy3\hskip.3em$\ssize##1$}\hss}\setbox6%
		\hbox{\hskip.6em\copy2\copy2$##1##1$}%
		\ifdim\MAXright<\wd6\MAXright\wd6\fi}%
	\def\whiterootdownright##1{\hbox to0pt{\raise-.5em\copy2\hskip-.25em\raise-1.35\ht1
			\hbox{\copy7\hskip.3em$\ssize##1$}\hss}\setbox6
		\hbox{\hskip.6em\copy2\copy2$##1##1$}
		\ifdim\MAXright<\wd6\MAXright\wd6\fi}
	\def\rootdown##1{\hbox to0pt{\hskip-.05em\vrule height.25em depth.65em%
			\hskip-.25em\raise-.95em\hbox{\copy3\hskip.3em$\ssize##1$}\hss}%
		\setbox6\hbox{$##1$}%
		\ifdim\MAXright<\wd6\MAXright\wd6\fi}%
	\def\whiterootdown##1{\hbox to0pt{\hskip-.05em\vrule height.25em depth.65em
			\hskip-.25em\raise-.95em\hbox{\copy7\hskip.3em$\ssize##1$}\hss}
		\setbox6\hbox{$##1$}
		\ifdim\MAXright<\wd6\MAXright\wd6\fi}
	\def\dots{\hskip.5em\cdots\hskip.5em}}%
\def\enddynkin{\ifdim\MAXright>1em\hskip.5\MAXright\else\hskip.5em\fi\egroup}%

\begin{document}

	\bigskip
	\title[Descriptions of strongly multiplicity free representations for simple Lie algebras]
	{Descriptions of strongly multiplicity free\\ representations for simple Lie algebras}

	\author{Bin-Ni Sun and Yufeng Zhao}

	\address[Sun]{Department of Mathematics, Peking University,
		Beijing, China}
		\address[Zhao]{Department of Mathematics, Peking University,
		Beijing, China}
		\email{ Zhaoyufeng@math.pku.edu.cn (Corresponding author)}
		\thanks{The research described in this paper is supported by grants No. 7100902512
		from Research Grant Council of China}

	\keywords{simple Lie algebra, strongly multiplicity free module, Casimir operator, invariant endomorphism algebra}
	\subjclass[2010]{Primary 22E47; Secondary 22E46}

	\begin{abstract}
		Let $\frg$ be a complex simple Lie algebra and $Z(\frg)$ be the center of the universal enveloping algebra $U(\frg)$. Denote by $V_\lambda$ the finite-dimensional irreducible $\frg$-module with highest weight $\lambda$.
		Lehrer and Zhang defined  the notion of strongly multiplicity free representations for simple Lie algebras motivated by studying the structure of the endomorphism  algebra  $\End _{U(\frg)}(V_\lambda^{\otimes r})$ in terms of the quotients of the Kohno's infinitesimal braid algebra. 
		Kostant  introduced the $\frg$-invariant endomorphism  algebras  $R_\lambda(\frg)= (\End V_\lambda\otimes U(\frg))^\frg$
		and
		$R_{\lambda,\pi}(\frg)=(\End V_\lambda\otimes \pi (U(\frg)))^\frg.$
		In this paper, we give some other criteria for a multiplicity free representation to be
		strongly multiplicity free by classifying the pairs $(\frg, V_\lambda)$, which are multiplicity free and for such pairs,
		$R_\lambda(\frg)$ and $R_{\lambda,\pi}(\frg) $ are generated by generalizations of the quadratic Casimir elements of $Z(\frg)$.
	\end{abstract}

	\maketitle

	\section{Introduction}\label{section intro}

	A finite-dimensional irreducible weight module for  a complex simple Lie algebra  is said to be  multiplicity free if all the weights of this module occur with  multiplicity one. The complete  classification of multiplicity free representations 
	appears in Howe's 1992 Schur Lecture [Ho].
	Lehrer and Zhang [LZ] introduced a particular kind of  multiplicity free modules named strongly multiplicity free modules for simple Lie algebras motivated by studying the structure of a kind of  endomorphism  algebras in terms of the quotients of Kohno's infinitesimal braid algebra.
	
	 In this paper, we give some other criteria for a multiplicity free module to be  strongly multiplicity free. 
	 Our point of view is formulated during the study of the structures of a kind of invariant endomorphism algebras introduced by Kostant [Ko1].

	To formulate our main theorems, we need the following definitions and assertions.
	Let $\frg$ be a complex simple Lie algebra and $Z(\frg)$ be the center of its universal enveloping algebra $U(\frg)$.
	One says that the representation $\pi: U(\frg)\rightarrow \End V$
	admits an infinitesimal character  and $\chi_{\pi}$ is that character if $\pi(u)$ reduces to a scalar operator $\chi_{\pi}(u)$ on $V$ for every $u \in Z(\frg)$.
	It is a theorem of Dixmier that any irreducible representation  of $U(\frg)$ admits an infinitesimal character.
	
	Fix a finite-dimensional irreducible representation $\pi_{\lambda}\colon \frg \rightarrow \End V_\lambda$ with highest weight $\lambda$. Let
	$$\pi_{\nu} \colon \frg \rightarrow \End V_\nu $$
	be an another  arbitrary  $\frg$-module having an infinitesimal central character $\chi_\nu$.	
	To study the infinitesimal characters occurring in the tensor product
	$V_\lambda\otimes V_\nu$, a celebrated theorem of Kostant [Ko1] says that the generalized infinitesimal characters occurring
	in the tensor product of $V_\lambda\otimes V_\nu$ are $\chi_{\nu + \lambda_i}$, where $\lambda_i$ are the weights of $V_\lambda$.
	In Kostant's proof,  the following $\frg$-invariant endomorphism  algebras  $$R_\lambda(\frg)= (\End V_\lambda\otimes U(\frg))^\frg$$
	and
	$$R_{\lambda,\nu}(\frg)=(\End V_\lambda\otimes \pi_{\nu}[U(\frg)])^\frg$$
	 play pivotal roles.
	These algebras were also investigated from a different perspective by Kirillov [Ki1,Ki2]
	as  `quantum family algebras'.
	It was proved  that
	$R_\lambda(\mathfrak{g})$ and $ R_{\lambda,\nu}(\mathfrak{g})$ are commutative if and only if $V_{\lambda}$ is a multiplicity free module  (cf. [P], [Ki1] and [Ki2]).
	The commutative $\frg$-invariant endomorphism algebra
	has many connections with equivariant cohomology [P].

	As  shown in [Ko2] there is a $\frg$-submodule $E$ of $U(\frg)$
	such that the multiplication
	$$ Z(\frg)\otimes E\rightarrow U(\frg)$$
	is a $\frg$-isomorphism.  It follows that  $R_\lambda(\frg)$ is a free $Z(\frg)$-module.

	Consider the map
	$$\delta_\lambda: U(\frg) \rightarrow \End V_\lambda\otimes U(\frg)$$
	defined by
	$$ \delta_\lambda(x)= \pi_{\lambda}(x)\otimes 1 + \textrm{id}\otimes x\text{ for } x \in \frg,$$
	which is extended to be a homomorphism of associative algebras.	

	Let $m$ be the dimension of $\mathfrak{g}$ and $B$ be the Killing form of $\frg$.  Assume $\{x_1, \cdots, x_m\}$ is a basis of $\frg$ and $\{x_1^*, \cdots, x_m^* \}$ is its dual basis with respect
	to $B$.  The second order Casimir element $C$ is defined by
	$$C=\sum\limits_{i=1}^{m}x_{i}x^{*}_{i}$$
	in $Z(\frg)$, and clearly it is independent of choice of the basis $x_i$.
	Then
	$$ \delta_\lambda(C)= \pi_{\lambda}(C)\otimes 1 +\sum\limits_{i=1}^{m}\pi_\lambda(x_{i})\otimes x^{*}_{i}+\sum\limits_{i=1}^{m}\pi_\lambda(x^*_{i})\otimes x_{i} + 1\otimes C.$$
Set
	$$M_\lambda(C)=\sum\limits_{i=1}^{m}\pi_{\lambda}(x_{i})\otimes x^{*}_{i}.$$
	It is easy to  check that $M_\lambda(C)$ is also independent of choice of the basis $x_i$,
	and thus it is equal to $\sum\limits_{i=1}^{m}\pi_\lambda(x^*_{i})\otimes x_{i}$.
	Then
	$$ \delta_\lambda(C)= \pi_{\lambda}(C)\otimes 1 + 2M_\lambda (C)+ 1\otimes C.$$
	\medskip
	For any other arbitrary irreducible 
	 representation
	 $\pi_{\nu}\colon \frg \rightarrow \End V_\nu$,
	  consider the map
	$$\delta_{\lambda,\nu}: U(\frg) \rightarrow \End V_\lambda\otimes  \End V_\nu$$
	defined by
	$$ \delta_{\lambda,\nu}(x)= \pi_{\lambda}(x)\otimes 1 +  1\otimes \pi_\nu(x) \text{ for } x\in \frg,$$
	which is extended to be a homomorphism of associative algebras.  
	Set $$M_{\lambda,\nu}(C)=\sum\limits_{i=1}^{m}\pi_{\lambda}(x_{i})\otimes \pi_{\nu} (x^{*}_{i}).$$	
	Then
	$$ \delta_{\lambda,\nu}(C)= \pi_{\lambda}(C)\otimes 1 + 2M_{\lambda,\nu}(C)+ 1\otimes \pi_\nu (C). $$
	
	\medskip
	
	Denote by $\Pi(V_{\lambda})$ the weight set of 
	$\frg$-module	$V_{\lambda}$.
	Recall from [LZ] that   a multiplicity free module	$V_{\lambda}$ is called strongly multiplicity free  if  for any pair of  weights $\mu, \nu \in \Pi(V_{\lambda})$, we have $\mu \leq \nu$ or $\nu \leq \mu$, i.e.  $\Pi(V_{\lambda})$ is linearly ordered under the partial ordering. 
	
	The following theorem  gives some equivalent descriptions for strongly multiplicity free
	modules.
	
	\medskip
	
	\noindent{\bf Theorem.}  {\it Assume $ V_\lambda$ is an
		irreducible multiplicity free  $\frg$-module with dimension $d_{\lambda}$. The following assertions are equivalent:

		$\operatorname{a)}$ $V_\lambda$ is a strongly multiplicity free $\mathfrak{g}$-module;

		$\operatorname{b)}$  The semisimple operator $M_{\lambda,\nu}(C)$ has distinct  eigenvalues  on distinct summands of $V_\lambda \otimes V_{\nu}$ for any other arbitrary highest weight module $V_{\nu}$;

		$\operatorname{c)}$ The endomorphism algebra $R_{\lambda,\nu}(\mathfrak{g})$ is a polynomial algebra in $M_{\lambda,\nu} (C)$ for any other arbitrary highest weight module $V_{\nu}$;

		$\operatorname{d)}$  $1,M_{\lambda} (C)\ldots, M_{\lambda}(C)^{d_\lambda-1}$ form a basis of free  $Z(\frg)$-module
		$R_\lambda(\mathfrak{g})$;
		
		$\operatorname{e)}$  $2\text{ht}(\lambda)=d_\lambda-1$;
		
		$\operatorname{f)}$   $V_\lambda$ is irreducible when restricted to a principal $\frsl_2$ in $\frg$.
	}

	\medskip

	\noindent{\bf Remark.}  {\it 
		The notation $\text{ht}(\lambda)$ is defined in Equation \eqref{height}.
	}
	
	\medskip
	
	The paper is organized as follows: In section 2, we will recall some basic classifications and properties related to multiplicity free and strongly multiplicity free representations for complex simple Lie algebras. And some facts about the Prouhet-Tarry-Escott problem in 
	Diophantine number theory are collected for the proof of the main theorem.
	In section 3, we will prove the assertions a), b) and c) are equivalent to each other.
	In section 4, we will prove that assertion e) is equivalent to assertion f) 
	by means of the principal three dimensional Lie subalgebra techniques; and d) implies e) by Poincar\'{e} polynomials of multiplicity free representations for simple Lie algebras. 
	We will also prove that a) implies d). In section 5, we will prove that assertion a) is equivalent to assertion b) and e) through case by case calculations. In the appendix, we will  present explicit formulas of eigenvalues of Casimir operators in terms of power sums; we will also give the M-type matrices for the strongly multiplicity free  modules.
	\medskip
	\section{Preliminaries}
	
	In this section, we will prepare  some basic properties related to infinitesimal characters and 
	strongly multiplicity free representations. We will also collect some facts about the Prouhet-Tarry-Escott problem in 
	Diophantine number theory for the proof of the main theorem.
	
	Let $\frg$ be a complex simple Lie algebra, $ \mathfrak{h} $ be a Cartan subalgebra, $\mathfrak{h}^{*} $ be the dual space to $ \mathfrak{h} $ and $\phi^{+}$ be a system of positive roots.	
	Assume  $\Pi =\{\alpha_{1}, \cdots,\alpha_{l}\}$ is the
	simple root system in $\phi^{+}$ and $\delta$ is the half sum of positive  roots and $W$ is  the Weyl group.  Denote $ \mathfrak{n}=\sum\limits_{\alpha \in \phi^{+}}\frg^{\alpha}$ and  $ \mathfrak{n}^{-}=\sum\limits_{-\alpha \in \phi^{+}}\frg^{\alpha}$. Let $\mathcal{P}_{+}$ be the set of dominant integral weights for $\mathfrak{g}$.	
	 The fundamental dominant weights are denoted by $\{\omega_1,\cdots, \omega_l\}$, i.e. $$\frac{2(\omega_{i},\alpha_{j})}{(\alpha_{j},\alpha_{j})}=\delta_{i,j}.$$
	 Let $\{h_{1}, \cdots, h_{l}\}$ be the dual basis of $\{\omega_1,\cdots, \omega_l\}$ in $\frh$, i.e.  $	 
	\omega_{i}(h_{j})=\delta_{i,j}$.
	 	
	Let $\leq$ denote the usual ordering induced on the weights by the positive roots, i.e. $\lambda \leq \mu$ if and only if $\mu-\lambda$ is a non-negative $\mathbb{Z}$-linear combination of positive roots.
	For $\lambda \in \mathcal{P}_{+}$, the irreducible $\mathfrak{g}$-module $V_{\lambda}$ with highest weight $\lambda$ is finite-dimensional. 
	Denote by $\Pi(V_{\lambda})$ the weight set of 
	$\frg$-module	$V_{\lambda}$.

		Consider the zero weight space of $U(\frg)$:
	$$U(\frg)_0=\{x\in U(\frg) \ | \ [h,x]=0, \ \forall \ h \in \frh\}.$$
	The vector space $$\mathfrak{K} =
	U(\frg)_0 \cap U(\frg)\mathfrak{n} =U(\frg)_0 \cap \mathfrak{n}^{-}U(\frg)$$ is a 2-sided ideal of $U(\frg)$
	and $U(\frg)_0= \mathfrak{K} \oplus U(\frh)$. Let $$\psi:U(\frg)_0\rightarrow U(\frh)$$ be the projection map obtained from this decomposition, which is called Harish-Chandra homomorphism.
	For any $\lambda \in \mathfrak{h}^{*}$, the infinitesimal central character is given by 
	$$\chi_{\lambda}(u)=\lambda(\psi(u)), \ \forall u \in Z(\mathfrak{g}).$$

	In particular, the eigenvalues of second order Casimir element $C$ is uniquely determined by $$\chi_{\lambda}(C)=(\lambda,\lambda+2\delta).$$	In the Appendix, we will give Popov's formulas related to  the eigenvalues of Casimir element $C$ in terms of power sum functions.
	
	Since the Lie subalgebra $ \mathfrak{h}$ is abelian, we can identify the universal enveloping algebra $U(\mathfrak{h})$ with the symmetric algebra $ S(\mathfrak{h}) $. Define the following twisting algebra homomorphism:
	$$\tau:U(\mathfrak{h})\rightarrow  U(\mathfrak{h});\tau(h_{i})=h_{i}-1, \  i =1,\cdots, l.$$
	We know that both $ \mathfrak{h}^*$ and $ \mathfrak{h}$ are Weyl group $W$-modules. And the $W$-actions on  $ \mathfrak{h}^*$ and $ \mathfrak{h}$ are related by:$$(\omega f)(h)=f(\omega^{-1}h), \forall \ \omega \in W, f \in  \mathfrak{h}^*, h \in  \mathfrak{h}.$$
	For any $\nu\in \mathfrak{h}^{*} $, denote the translation map
	$$\tau_{\nu}:
	\mathfrak{h}^{*} \rightarrow \mathfrak{h}^{*} ; \lambda\mapsto  \lambda+\nu, \forall   \  \lambda
	\in          \mathfrak{h}^{*}.$$ 		
Now the translated Weyl group $$\widetilde{W}=\{\tilde{\sigma}=\tau_{-\delta}\sigma \tau_{\delta}  | \ \sigma \in W\}$$ acts on
	$U( \mathfrak{h})$ via:
	$$\lambda(\tilde{\sigma}(h))=[\tilde{\sigma}^{-1}(\lambda)](h)=[{\sigma}^{-1}(\lambda+\delta)-\delta](h), \forall \ \sigma \in W, \lambda \in  \mathfrak{h}^*, h \in  \mathfrak{h}.$$By Theorem 11.30 in Carter's book [C], we know that the map $\tau \psi$ gives an isomorphism of algebras $Z(\frg)\rightarrow S(\frh)^{W}$. Then we have:
	$$[\sigma(\lambda)-\delta](\psi(z))=(\lambda-\delta)(\psi(z)), \ \forall \ \sigma \in W, \lambda \in \mathfrak{h}^*, z \in Z(\frg).$$Hence,
	$$\tilde{\sigma}(\psi(z))=\psi(z), \ \forall \ \sigma \in W, z\in Z(\frg).$$	
	 Therefore,
	the map $z\mapsto \psi(z)$ is an algebra isomorphism from $Z(\frg)$ to  $U( \mathfrak{h})^{\widetilde{W}}$ (cf. Theorem 2.2 in [Ko1]).

	\begin{defi}
		An irreducible $\mathfrak{g}$-module $V_{\lambda}$ is called multiplicity free if each weight space	of	
		$V_{\lambda}$ has dimension one.
	\end{defi}

	The following  list  is perhaps well-known to experts.  It appears in Howe's 1992 Schur Lecture [Ho]:

	\begin{thm}
		Let $\mathfrak{g}$ be a complex simple Lie algebra. The following is a list of the  multiplicity free irreducible $\mathfrak{g}$-modules: 
		
		(1) $A_{l}$-module $V_{\omega_{k}}$ ( $ k=1,\ldots, l$),
		$V_{k\omega_1}, V_{k\omega_l}, k=1,2,\ldots$;
		
		(2) $B_{l}$-module
		$(l\geq 2) V_{\omega_1}, \omega_l  \text{ (spin representation)}$;
		
		(3) $C_{l}$-module 
		$(l\geq 3) V_{\omega_1}$, $C_3$-module $V_{\omega_3}$;

		(4)	$D_l (l\geq 4)$-module $V_{\omega_1}$, $V_{\omega_{l-1}}$,  $V_{\omega_l}  \text{ (spin representations)}$;

		(5) $E_6$-module $V_{\omega_1}$, $V_{\omega_6}$ ($\dim V_{\omega_{1}} = \dim V_{\omega_{6}} = 27$);
		
		(6) $E_7 $-module $V_{\omega_7}$ ($\dim V_{\omega_{7}} = 56$);
		
		(7) $G_{2}$-module $V_{\omega_{1}}$ ($\dim V_{\omega_{1}} = 7$).
		
	\end{thm}

	\begin{defi}
		The irreducible multiplicity free $\frg$-module	$V_{\lambda}$ is called strongly multiplicity free  if for any pair of  weights $\mu, \nu \in \Pi(V_{\lambda})$, we have $\mu \leq \nu$ or $\nu \leq \mu$, i.e.  $\Pi(V_{\lambda})$ is linearly ordered under the partial ordering.
	\end{defi}

	\begin{thm}
		Let $\mathfrak{g}$ be a complex simple Lie algebra. The following is a complete list of the strongly multiplicity free irreducible $\mathfrak{g}$-modules [LZ]: 
		
		(1) $A_{l}$-module $V_{\omega_{1}}$ and its dual $V_{\omega_{l}}$ for $l \geq 2$; 
		
		(2) $B_{l}$-module $V_{\omega_{1}}$ for $l \geq 2$;
		
		(3) $C_{l}$-module $V_{\omega_{1}}$ for $l \geq 3$;
		
		(4) $G_{2}$-module $V_{\omega_{1}}$;
		
		(5) all irreducible $\mathfrak{sl}_{2}$-modules of dimension greater than $1$. 
	\end{thm}

	\begin{prop}\label{smf_has_distinct_eigs}
		Assume $V_{\lambda}$ is strongly multiplicity free and $\nu \in \mathcal{P}_{+}$. For the weights $\lambda_{i}, \lambda_{j} \in \Pi(V_{\lambda})$ such that $\nu + \lambda_{i}$ and $\nu + \lambda_{j} \in \mathcal{P}_{+}$, we have $\chi_{\nu + \lambda_{i}}(C) = \chi_{\nu + \lambda_{j}}(C)$ if and only if $\lambda_{i} = \lambda_{j}$.
	\end{prop}

	\begin{proof} 
		Let $\varLambda_{i} = \nu + \lambda_{i}$ and $\varLambda_{j} = \nu + \lambda_{j}$. Assume $\chi_{\varLambda_{i}}(C)= \chi_{\varLambda_{j}}(C)$. Then
		$$
		0 = \chi_{\varLambda_{i}}(C) - \chi_{\varLambda_{j}}(C) = (\varLambda_{i} + \varLambda_{j} + 2\delta, \lambda_{i} - \lambda_{j}).
		$$
		Since $\varLambda_{i} + \varLambda_{j} + 2\delta \in \mathcal{P}_{+}$ and $\lambda_{i} - \lambda_{j}$ is either a sum of positive or negative roots of $\mathfrak{g}$, it follows that $\lambda_{i} = \lambda_{j}$. The sufficient condition is obvious.
	\end{proof}

	We observe that some multiplicity free but not strongly multiplicity free weights are minuscule.
	Kass [K] gave an explicit $\mathfrak{g}$-module decomposition related to  minuscule modules as follows:
	
	\begin{prop}
		Assume $\lambda$ is a minuscule weight and $\nu \in \mathcal{P}_{+}$, then
		$$V_{\lambda} \otimes V_{\nu} = \sum_{\substack{\lambda_{i} \in \Pi(V_{\lambda}) \\ \nu + \lambda_{i} \in \mathcal{P}_{+}}} V_{\nu + \lambda_{i}}.$$
	\end{prop}

	The  Prouhet-Tarry-Escott problem (written briefly as PTE problem) consists of finding two lists of integers $[x_{1}, x_{2}, \dots, x_{n}]$ and $[y_{1}, y_{2}, \dots, y_{n}]$,  distinct up to permutation, such that
	$$
	\sum_{i = 1}^{n}x_{i}^{j} = \sum_{i = 1}^{n}y_{i}^{j}, \qquad j = 1, 2, \dots, m.
	$$
	We call $n$ the size of the solution and $m$ the degree. 
	A trivial solution of PTE problem means that the $x$'s merely form a permutation of the $y$'s.

	\begin{lemma} \label{PTE_exist}
		For a non-trivial solution to exist, then $n \geq m+1$ (cf. [DB]).
	\end{lemma}

	\begin{lemma} \label{solution_of_degree_2}
		For $m = 2$ and $n = 3$ in PTE problem, $[s+1, s-3, s-4]$ and $[s, s-1, s-5]$ is a non-trivial integer solution for any $s \in \mathbb{Z}$.
	\end{lemma}

	\section{$\frg$-invariant endomorphism  algebras}	
	
	In this section, we will  present some  investigations about $\frg$-invariant endomorphism algebras which were introduced by Kostant [Ko1].	
	We will also prove that assertions a), b) and c) in the main theorem are equivalent to each other.

	For $\lambda, \nu \in \mathcal{P}_{+} $, consider the associative algebras
	$$R_\lambda(\mathfrak{g}) = (\End V_\lambda\otimes U(\frg))^\frg$$
	and
	$$R_{\lambda,\nu}(\mathfrak{g}) = (\End V_\lambda\otimes \End V_\nu)^\frg.$$

	Kostant [Ko2] showed that there is a $\frg$-submodule $E$ of $U(\frg)$
	such that the multiplication
	$$ Z(\frg)\otimes E\rightarrow U(\frg)$$
	is a $\frg$-module isomorphism.  
	It follows that  $U(\frg)$ and $R_\lambda(\mathfrak{g})$ are free $Z(\frg)$-modules.

	There is a particular kind of elements  in  $R_\lambda(\mathfrak{g})$ and $R_{\lambda,\nu}(\mathfrak{g}) $, which can be thought of generalizations of the quadratic Casimir elements in $Z(\frg)$.
	We are concerned with the classification of  such pairs $(\frg, V_\lambda)$ so that  $R_\lambda(\mathfrak{g})$ and $R_{\lambda,\nu}(\mathfrak{g})$ are generated by such elements.
	To formulate explicitly,  we consider the map
	$$\delta_\lambda: U(\frg) \rightarrow \End V_\lambda\otimes U(\frg)$$
	defined by
	$$ \delta_\lambda(x)= \pi_{\lambda}(x)\otimes 1 +  1\otimes x  \text{ for } x \in \frg,$$
	which is extended to be a homomorphism of associative algebras.

	Let $m$ be the dimension of $\mathfrak{g}$, $x_i$ be a basis of $\frg$ and $x_i^*$ be the dual basis with respect to the Killing form $B$.  The second order Casimir element $C$ is defined by
	$$C=\sum\limits_{i=1}^{m}x_{i}x^{*}_{i}$$
	in $Z(\frg)$.
	It follows that
	$$ \delta_\lambda(C)= \pi_{\lambda}(C)\otimes 1 +\sum\limits_{i=1}^{m}\pi_\lambda(x_{i})\otimes x^{*}_{i}+\sum\limits_{i=1}^{m}\pi_\lambda(x^*_{i})\otimes x_{i} + 1\otimes C.$$
	We set
	$$M_\lambda(C)=\sum\limits_{i=1}^{m}\pi_{\lambda}(x_{i})\otimes x^{*}_{i}.$$
		According to Proposition 11.33 in [C], the second order Casimir element $C$ is independent of the choice of the basis $x_i$. In the same way, we can  check that $M_\lambda(C)$ is also independent of the choice of the basis $x_i$,
	and thus it is equal to $\sum\limits_{i=1}^{m}\pi_\lambda(x^*_{i})\otimes x_{i}$.
	Then
	\begin{equation} \label{M_{lambda}(C)}
		M_{\lambda}(C) = \frac{1}{2}[\delta_\lambda(C) - \pi_{\lambda}(C)\otimes 1 - 1\otimes C].
	\end{equation}
	\medskip
	Consider the map
	$$\delta_{\lambda,\nu}: U(\frg) \rightarrow \End V_\lambda\otimes  \End V_\nu$$
	defined by
	$$ \delta_{\lambda,\nu}(x)= \pi_{\lambda}(x)\otimes 1 +  1\otimes \pi_\nu(x) \text{ for } x\in \frg,$$
	which is also extended to be a homomorphism of associative algebras.  
	We set
	$$M_{\lambda,\nu}(C)=\sum\limits_{i=1}^{m}\pi_{\lambda}(x_{i})\otimes \pi_{\nu} (x^{*}_{i}).$$
	Similarly, we observe
	\begin{equation} \label{M_{lambda, nu}(C)}
		M_{\lambda,\nu}(C) = \frac{1}{2}(\delta_{\lambda,\nu}(C) - \pi_{\lambda}(C)\otimes 1 - 1\otimes \pi_\nu (C)).
	\end{equation}
	Since $C \in Z(\mathfrak{g})$, we have $\delta_\lambda(C) \in R_{\lambda}(\mathfrak{g})$ and $\delta_{\lambda,\nu}(C) \in R_{\lambda, \nu}(\mathfrak{g})$.
	By \eqref{M_{lambda}(C)} and \eqref{M_{lambda, nu}(C)}, we also get $M_\lambda(C) \in R_{\lambda}(\mathfrak{g})$ and $M_{\lambda,\nu}(C) \in R_{\lambda,\nu}(\mathfrak{g})$. 
	For convenience, these  matrices are called  M-type matrices.

	Assume $\mathscr{P}(\mathfrak{h}^{*})$ is the polynomial algebra on $\mathfrak{h}^{*}$.  Recall that  $S(\mathfrak{h})$ is the symmetric algebra on $\mathfrak{h}$ and $U(\mathfrak{h})$ is the universal enveloping algebra of $\mathfrak{h}$.
	Let $\lambda_{1}, \dots, \lambda_{k}$ be the distinct weights of $V_{\lambda}$.
	For each $i \in \{1, \dots, k\}$, the correspondence 
	$$\mathfrak{h}^{*} \to \mathbb{C}$$
	$$\nu \mapsto f_{C,i}(\nu) = \frac{1}{2}(\chi_{\nu +\lambda_{i}}(C) - \chi_{\lambda}(C) - \chi_{\nu}(C))$$ 
	determines a polynomial function on $\mathfrak{h}^{*}$. 
    We have $f_{C, i} \in \mathscr{P}(\mathfrak{h}^{*}) \simeq S(\mathfrak{h}) \simeq U(\mathfrak{h})$.

	We recall some basic properties related to  M-type matrices given by Kostant (cf. Theorem 4.9 in [Ko1], Theorem 6 in [G]).
		\begin{thm} 	Let $\{\lambda_1,\cdots , \lambda_k\}$ be the set of distinct weights of $ V_{\lambda}$.		
		There exists a monic polynomial $P_{\lambda}(X)$ of degree $k$ with coefficients in $Z(\frg)$,
		such that $P_{\lambda}(X)$ is the minimal polynomial of $M_\lambda(C)$.
	\end{thm}

	\begin{prop} Any symmetric combinations of $f_{C,1}(\nu),\cdots, f_{C,k}(\nu)$ determines an invariant polynomial function in $U( \mathfrak{h})^{\widetilde{W}}$. 		
	\end{prop}

	\begin{proof}
		It follows from Kostant (cf. Theorem 2.5 in [Ko1]) that for $\nu_{1}, \nu_{2} \in \mathfrak{h}^{*}$ infinitesimal characters $\chi_{\nu_{1}} = \chi_{\nu_{2} }$ if and only if $\nu_{1}$ and $\nu_{2} $ are translated Weyl group $\tilde{W}$-conjugate.
		Besides, by Lemma 4.9 in [Ko1] we have $\tilde{\sigma}(\nu_{1} + \nu_{2}) = \tilde{\sigma}(\nu_{1}) + \sigma(\nu_{2})$ for $\nu_{1}, \nu_{2} \in \mathfrak{h}^{*}$ and $\sigma \in W$.
		The translated Weyl group $\widetilde{W}$ acts on the  polynomial function $f_{C,i}$ by the following way:
		$$\tilde{\sigma}f_{C,i}(\nu) = f_{C,i}(\tilde{\sigma}^{-1}(\nu)) = f_{C,i}({\sigma}^{-1}(\nu+\delta)-\delta)$$
		$$=\frac{1}{2}(\chi_{\tilde{\sigma}^{-1}(\nu) + \lambda_{i}}(C) - \chi_{\lambda}(C) - \chi_{\tilde{\sigma}^{-1}(\nu)}(C))$$
		$$=\frac{1}{2}(\chi_{\nu +\sigma(\lambda_{i})}(C)-\chi_{\lambda}(C)-\chi_{\nu}(C)).$$
		
		\medskip
		
		Since the Weyl group permutes the distinct weight of  $V_{\lambda}$, there exists an permutation $\varepsilon$ of the symmetric group $S_{k}$ such that $\sigma(\lambda_i)=\lambda_{\varepsilon(i)}$.
		Therefore, $$\tilde{\sigma}f_{C,i}(\nu)=f_{C,\varepsilon(i)}(\nu).$$
		Then, any symmetric combinations of  $f_{C,1}(\nu),\cdots, f_{C,k}(\nu)$ determines a invariant polynomial function in $U( \mathfrak{h})^{\widetilde{W}}$. 
	\end{proof}

	It follows from Gould [G] that the eigenvalues of  $M_{\lambda,\nu}(C)$ on the tensor module $ V_\lambda\otimes  V_\nu$ are of the form $f_{C,i}(\nu)$ for $\lambda_{i} \in \Pi(V_{\lambda})$ and $\nu + \lambda_{i} \in \mathcal{P}_{+}$.

	\begin{thm} \label{a = b}
		Assume $V_\lambda$ is an irreducible multiplicity free $\mathfrak{g}$-module. 
		Then $V_{\lambda}$ is strongly multiplicity free if and only if semisimple operator $M_{\lambda,\nu} (C)$ has distinct eigenvalues on distinct summands of $V_\lambda \otimes V_{\nu}$ for each dominant integral weight $\nu$.
		
	\end{thm}

	\begin{proof}By Proposition \ref{smf_has_distinct_eigs} and the formula of the eigenvalues of  $M_{\lambda,\nu}(C)$ on  the tensor module $ V_\lambda\otimes  V_\nu$, we see the necessary condition is valid.
	We will leave the detailed calculations for the proof of the sufficient condition  in  Section 5 (cf. Proposition 5.1, 5.3, 5.5, 5.9, 5.11, 5.12, 5.16, 5.18).
	 
	\end{proof}


	\begin{thm} 
		Assume $V_\lambda$ is an irreducible multiplicity free $\mathfrak{g}$-module and $\nu \in \mathcal{P}_{+} $.
		Then semisimple operator $M_{\lambda,\nu}(C)$ has distinct eigenvalues on each summand of $V_\lambda \otimes V_\nu$ if and only if $1,M_{\lambda, \nu} (C), \ldots, M_{\lambda, \nu}(C)^{s-1}$ form a basis of  $R_{\lambda, \nu}(\mathfrak{g})$ over $\mathbb{C}$, where $s$ is the number of irreducible components of $V_{\lambda} \otimes V_{\nu}$.
	\end{thm}
	
	\begin{proof} 
		Let $\Pi(V_{\lambda})=\{\lambda_1,\cdots , \lambda_k\}$ be the weight set of multiplicity free $\mathfrak{g}$-module $ V_{\lambda}$.		
		From Kostant (cf. Theorem 4.3 in [Ko1]), we have the $\mathfrak{g}$-module decomposition
		$$V_{\lambda}\otimes V_{\nu} = \sum_{\substack{i \in \{1, \cdots, k\} \\ \nu + \lambda_{i} \in \mathcal{P}_{+}}} n_{i} V_{\nu+\lambda_{i}},$$			
		where $n_{i} =0 $ or $1$.
		Let the index set $I = \{i \in \{ 1, \cdots, k \} | \nu + \lambda_{i} \in \mathcal{P}_{+} \text{ and } n_i=1\}$ and denote the number of elements of $I$ by $s$.
		
		If the semisimple operator $M_{\lambda,\nu}(C)$ has distinct  eigenvalues  on distinct summands of $ V_\lambda\otimes  V_\nu$, we can construct the following projection operators $P_{\nu + \lambda_{i}} : V_{\lambda} \otimes V_{\nu} \to V_{\nu + \lambda_{i}}$ by:
		\begin{equation} \label{P_{nu+lambda_{i}}}
			P_{\nu+\lambda_{i}} = \prod_{ j \in I \setminus \{i\}} \frac{M_{\lambda,\nu}(C) - f_{C,j}(\nu)}{f_{C,i}(\nu) - f_{C,j}(\nu)},
		\end{equation}
		for $i \in I$.
		By Schur's lemma and Equation \eqref{P_{nu+lambda_{i}}}, the $\mathfrak{g}$-invariant endomorphism algebra 
		$$ R_{\lambda,\nu}(\mathfrak{g}) = (\End V_\lambda \otimes \End V_\nu)^{\mathfrak{g}} \simeq \End_{\mathfrak{g}} (V_{\lambda} \otimes V_{\nu}) = \mathbb{C}_{i\in I}\{P_{\nu+\lambda_{i}}\}$$
		is a polynomial algebra generated by $M_{\lambda,\nu}(C)$ and $1,M_{\lambda, \nu} (C), \ldots, M_{\lambda, \nu}(C)^{s-1}$ form a basis of $R_{\lambda, \nu}(\mathfrak{g})$ over $\mathbb{C}$.
		
	    Assume $1, M_{\lambda, \nu} (C), \ldots, M_{\lambda, \nu}(C)^{s-1}$ form a basis of $R_{\lambda, \nu}(\mathfrak{g})$ over $\mathbb{C}$. Since  $P_{\nu + \lambda_{i}} \in  R_{\lambda, \nu}(\mathfrak{g})$ for each $i \in I$,  there exists certain non-zero polynomials $g_{i}(x) = a_{i, 0} + a_{i, 1}x + \cdots + a_{i, s-1}x^{s-1} \in \mathbb{C}[x]$ such that $P_{\nu + \lambda_{i}} =g_{i}(M_{\lambda, \nu}(C))$.
		Then we get
		$$g_{i}(f_{C, j}(\nu)) = 0 \ \text{ for } j \in I \setminus \{i\}, \quad  g_{i}(f_{C, i}(\nu)) = 1,$$
		where $f_{C, i}(\nu), i \in I$ are the eigenvalues of $M_{\lambda, \nu}(C)$.
	    Write the  index set $I = \{i_{1}, i_{2}, \dots, i_{s}\}$.
		Then the identity matrix $ I_{s}$ of order $s$ is equal to
		$$
		\begin{pmatrix}
			a_{i_{1}, 0} & a_{i_{1}, 1} & \cdots & a_{i_{1}, s-1} \\
			a_{i_{2}, 0} & a_{i_{2}, 1} & \cdots & a_{i_{2}, s-1} \\
			a_{i_{3}, 0} & a_{i_{3}, 1} & \cdots & a_{i_{3}, s-1} \\
			\vdots & \vdots & \ddots & \vdots \\
			a_{i_{s}, 0} & a_{i_{s}, 1} & \cdots & a_{i_{s}, s-1} 
		\end{pmatrix}
		\begin{pmatrix}
			1 & 1 & \cdots & 1 \\
			f_{C, i_{1}}(\nu) & f_{C, i_{2}}(\nu) & \cdots & f_{C, i_{s}}(\nu) \\
			f_{C, i_{1}}(\nu)^{2} & f_{C, i_{2}}(\nu)^{2} & \cdots & f_{C, i_{s}}(\nu)^{2} \\
			\vdots & \vdots & \ddots & \vdots \\
			f_{C, i_{1}}(\nu)^{s-1} & f_{C, i_{2}}(\nu)^{s-1} & \cdots & f_{C, i_{s}}(\nu)^{s-1}
		\end{pmatrix}  
		$$
		Comparing determinants on both sides of the above  equation, it is easy to see that the Vandermonde matrix is invertible and thus all the eigenvalues $f_{C, i}(\nu), i \in I$ of M-type matrix $M_{\lambda, \nu}(C)$ are distinct.
	\end{proof}

	\section{Principal $\mathfrak{sl}_{2}$ subalgebras}
	
	In this section, we will provide a criterion for a multiplicity free module to be  strongly multiplicity free from the view of the structures  restricted to its  principal three-dimensional subalgebra. Furthermore, we will  prove that  conditions a), d), e) and  f) in the main theorem are equivalent to each other.

	Recall that a principal $\frsl_2$ in $\frg$ is a three-dimensional subalgebra (written briefly as TDS) spanned by a triple $\{X,H,Y\}$ in $\frg$ such that
	$$[H,X]=2X, \ [H,Y]=-2Y, \ [X,Y]=H $$
	and the orbit of $X$ under the adjoint group of $\frg$ is the principal nilpotent orbit.

	Define $$h^{0}=\sum_{\alpha \in \phi^{+}}h_{\alpha},$$where  $h_{\alpha}$ is the coroot of $\alpha$,
	i. e. $h_{\alpha}$ is uniquely determined by the properties $h_{\alpha} \in [\mathfrak{g}^{\alpha}, \mathfrak{g}^{-\alpha}]$ and $\alpha(h_{\alpha}) = 2$.
	Let $h_i$ be the coroot of simple root  $\alpha_{i}$. 
	Then $\alpha_{i}(h^{0})=2, \ i=1,\cdots, l$.
	Furthermore, there are positive integers $c_i$ such that $h^{0}=\sum\limits_{i=1}^{l}c_{i}h_{i}$.	
	Let $\{E_i, F_i, H_i\}$ be a TDS with $E_i \in \frg_{\alpha_{i}}$, and 
	$F_i \in \frg_{-\alpha_{i}}$. Define $$e^{0}=\sum\limits_{i=1}^{l}E_{i}, \
	f^{0}= 
	\sum\limits_{i=1}^{l}c_{i}F_{i}.
	$$
	One has the commutation relations
	$$[h^{0}, e^{0}]= 2e^{0}, \
	[h^{0}, f^{0}] = -2f^{0}, \
	[e^{0}, f^{0}] = h^{0}.
	$$
	Then the three-dimensional Lie algebra spanned by $(h^{0}, e^{0}, f^{0})$ is a principal TDS and is denoted by $\mathfrak{a}^{0}$ (cf. Lemma 3.2.16 in [GW]). 
	Since $\mathfrak{a}^{0}$-module $V_{\lambda}$ is semisimple, $V_{\lambda}$ can be decomposed into a direct sum of irreducible $\mathfrak{a}^{0}$-modules. 
	
	For $\lambda \in \mathcal{P}_{+}$, write $\lambda=\sum\limits_{i=1}^{l}a_i\alpha_i$, the height of $\lambda$ is defined to be:
	\begin{equation} \label{height}
		\mbox{ht}(\lambda)=\sum\limits_{i=1}^{l}a_i.
	\end{equation}
	
	\begin{lemma}The $\mathfrak{a}^{0}$-module $V_{\lambda}$ is irreducible if and only if 
		$\dim V_{\lambda}=2\mbox{ht}(\lambda)+1$.
	\end{lemma}
	
	\begin{proof} First we note that   $$\lambda(h^{0})=
		\sum\limits_{i=1}^{l}a_{i}\alpha_i(h^{0})=2\mbox{ht}(\lambda) $$ is a positive integer.  
		For a nonzero vector $v_0$ of weight $\lambda$ and let $V^0$ be the $\mathfrak{a}^{0}$ cyclic subspace generated by $v_0$. Let $\sigma$ be the representation of $\mathfrak{a}^{0}$ on $V^0$ obtained from the restriction of 
		$\pi_{\lambda}$. Then $(\sigma,V^0)$ is an irreducible representation  of dimension $\lambda(h^{0})+1$. It follows that  $\mathfrak{a}^{0}$-module $V_{\lambda}$ is irreducible if and only if $\dim V_{\lambda}=\lambda(h^{0})+1=2\mbox{ht}(\lambda)+1$.
	\end{proof}

	\begin{thm}
		Assume  $\mathfrak{g}$-module $V_{\lambda}$ is multiplicity free for $\lambda \in \mathcal{P}_{+}$. Then $V_{\lambda}$ is strongly multiplicity free if and only if $V_{\lambda}$ is irreducible when restricted to $\mathfrak{a}^{0}$ in $\mathfrak{g}$.
	\end{thm}
	
	\begin{proof}
		Through case by case study, we find $\dim V_{\lambda}=2\mbox{ht}(\lambda)+1$ if and only if	
		$V_{\lambda}$ is strongly multiplicity free. We will  leave the detailed calculations in Section 5 (cf. Proposition 5.4, 5.6, 5.10, 5.15, 5.17, 5.19).
	\end{proof}

	Let $\mathscr{P}(\frg)$ be the algebra of all polynomials on $\frg$ and 
	$\frh$ be a Cartan subalgebra of $\frg$. 
	Let $S(\mathfrak{g}^{*})$ be the symmetric algebra over the dual space $\mathfrak{g}^{*}$.
	Then the Killing isomorphism $\mathscr{K}:\frg\rightarrow \frg^{*}$ can be extended to a $\frg$-module isomorphism $$\tilde{\mathscr{K}}:S(\frg)\rightarrow \mathscr{P}(\frg)=S(\frg^{*}).$$
	Due to the ad-invariance of the Killing form, we get $\tilde{\mathscr{K}}(S(\frg)^{\frg})=\mathscr{P}(\frg)^{\frg}$. We  denote these isomorphic algebras $S(\frg)^{\frg}\simeq \mathscr{P}(\frg)^{\frg}\simeq Z(\frg)$ by $J$ in common.

	\begin{defi} For $\nu \in \mathcal{P}_{+}$, 
		the space $$J_{\nu}(\mathfrak{g})=(V_{\nu}\otimes {\mathscr{P}}(\frg))^{\frg}$$
		is called the module of covariants of type $\nu$. 
	\end{defi}

	It was shown  by Kostant [Ko2] that  $J_{\nu}(\mathfrak{g})$ is a free  $J$-module
	with rank $ m_{\nu}^{0}$, where $m_{\nu}^{0}$ is the dimension of zero-weight space of $V_{\nu}$.
	Let $\{f_1 ,\cdots , f_l\}$ be  a set of basic invariants of $\mathscr{P}(\mathfrak{g})^{\frg}$ 
	and let $\{e_1,\cdots,e_l\}$ be their corresponding degrees. Then
	$$J_{\nu}(\mathfrak{g})=(f_1 ,\cdots , f_l)J_{\nu}(\mathfrak{g})\oplus H_{\nu},$$where $H_{\nu}$ is a  graded finite-dimensional vector space over $\mathbb{C}$.
	Any homogeneous $\mathbb{C}$-basis for $H_{\nu}$ is also  a free basis for the free $\mathscr{P}(\mathfrak{g})^{\frg}$-module $J_{\nu}(\mathfrak{g})$.
	Denote $$J_{\nu}(\mathfrak{g})=\oplus_{n\in \mathbb{N}}J_{\nu}(\mathfrak{g})_{n}; \  J_{\nu}(\mathfrak{g})_{n}=(V_{\nu}\otimes {\mathscr{P}}(\frg)_{n})^{\frg}.$$
	Then the Poincar\'{e} series of 
	$J_{\nu}(\mathfrak{g})$ is a rational function of the form:
	$$\mathcal{F}(J_{\nu}(\mathfrak{g}),q)=\sum_{n\in \mathbb{N}}\mbox{dim}J_{\nu}(\mathfrak{g})_{n}q^n=
	\frac{\sum_{j=1}^{m_{\nu}^{0}}q^{e_j(\nu)}}	
	{\prod_{j=1}^{l}(1-q^{e_j})}.$$The numbers $\{e_j(\nu) \ | \ 1 \leq j \leq m_{\nu}^{0}\}$, which are merely the degrees of a set of free homogeneous generators of $J_{\nu}(\mathfrak{g})$, are called the generalized exponents for $V_{\nu}$ [Ko2].

	Assume $\lambda^*$ is the highest weight of the dual module of $V_{\lambda}$.	
	Given any weight $\mu$ of $V_\lambda$, let $V_{\lambda}(\mu)$ be the corresponding weight space.
	We say that $\mu$ is on the $n$-th floor if
	$\mu-(-\lambda^*)=\sum\limits_{i=1}^{l}n_{i}\alpha_i$ with $\sum\limits_{i=1}^{l}n_{i}=n$. 	
	Set $$(V_{\lambda})_{n}=\sum\limits_{\mu: \mbox{floor}(\mu)=n}V_{\lambda}(\mu).$$
	It was pointed by  Panyushev that  the  Dynkin polynomial is: $$D_{\lambda}(q)=\sum\limits_{n}(\mbox{dim}(V_{\lambda})_{n})q^n=\sum_{\mu \in \Pi(V_{\lambda})} m_{\lambda}^{\mu}q^{\mbox{ht}(\lambda+\mu)}$$
	and its degree is $2\mbox{ht}(\lambda)$ [P].	
	Due to Panyushev (cf. Proposition 5.2 in [P]), the Poincar\'{e} series for $(\End V_{\lambda}\otimes {\mathscr{P}}(\frg))^{\frg}$ is equal to $D_{\lambda}(q)\prod\limits_{i=1}^{r}\frac{1}{1-q^{e_i}}$
	if and only if $\frg$-module $V_\lambda$ is multiplicity free.
	Since the  $J$-modules $(\End V_{\lambda}\otimes {\mathscr{P}}(\frg))^{\frg}$
	and $R_\lambda(\frg)$ are isomorphic, the  Poincar\'{e} series for $R_\lambda(\frg)$ and  $(\mbox{End}V_{\lambda}\otimes {\mathscr{P}}(\frg))^{\frg}$ are equal.

	The following  symbol method related to the filtered module introduced by 
	Bernstein and Lunts (cf. Lemma 4.2 in [BL]) is important for dealing with
	the filtered $Z(\frg)$-module $R_\lambda(\frg)$.
	
	\begin{lemma} 
		Let $M=\oplus_{i}M_i$ be a filtered module filtered by $F^{i}M=\oplus_{j\leq i} M_j$ over a filtered algebra $A$.
		Denote $\mbox{gr}M=\oplus_{i}\mbox{gr}_{i}M$, where $\mbox{gr}_{i}M=F^{i}M/
		F^{i-1}M$. For every $m \in M$, we denote $\sigma(m)$ its symbol in  $\mbox{gr}M$. Let $\{m_{\alpha}\}$ be family of elements of $M$. Suppose that their symbols $\{\sigma(m_{\alpha})\}$ form a free basis for the $\mbox{gr}A$-module $\mbox{gr}M$. Then  $\{m_{\alpha}\}$ is a free basis of the $A$-module $M$.
	\end{lemma}

	\begin{thm}
		The matrices 	
		$1,M_{\lambda}(C), \ldots, M_\lambda(C)^{d_\lambda-1} $ form a free basis for $Z(\frg)$-module $R_\lambda(\frg)$ if and only if $\mathfrak{g}$-module  $V_\lambda$ is strongly multiplicity free.
	\end{thm}
	
	\begin{proof}  Assume $\mathfrak{g}$-module  $V_\lambda$ is strongly multiplicity free. By Kirillov's theorems (cf. Theorems A, B, C, G, $\mathcal{Q}$ in [Ki1]), we know that $1,M_{\lambda}(C), \ldots, $ $M_\lambda(C)^{d_\lambda-1} $ form a basis for  $J$-module $(\End V_{\lambda}\otimes S(\frg))^{\frg}$. Since the associated graded algebra of the filtered algebra $U(\frg)$ is isomorphic to $S(\frg)$ (cf. Proposition 11.1 in [C]), we obtain that  the associated graded module of the filtered $Z(\frg)$-module $R_\lambda(\frg)$ is isomorphic to  $J$-module $(\End V_{\lambda}\otimes S(\frg))^{\frg}$. By the symbol method related to the filtered module introduced by Bernstein and Lunts (cf. Lemma 4.2 in [BL]),
		we obtain that  $1,M_{\lambda}(C), \ldots, M_\lambda(C)^{d_\lambda-1} $ form a free basis for $Z(\frg)$-module $R_\lambda(\frg)$ .

		On the other hand, if $1,M_{\lambda}(C), \ldots, M_\lambda(C)^{d_\lambda-1}$ form a free  basis for $Z(\frg)$-module $R_\lambda(\frg)$, then the generalized exponents of $\End V_\lambda$
		is $0,1,\ldots,d_\lambda-1$.
		It follows from the above results about  the Poincar\'{e} series for
		$R_\lambda(\frg)$ that the largest possible exponent of $\End V_\lambda$ is $2\text{ht}(\lambda)$. Therefore, $2\text{ht}(\lambda)=d_\lambda-1$, which is equivalent to
		that $V_\lambda$ is an irreducible module for the principal TDS  $\mathfrak{a}^{0}$ by Lemma 4.2. Hence, Theorem 4.3 implies that 
		$\mathfrak{g}$-module  $V_\lambda$ is strongly multiplicity free.
	\end{proof}

	\section{$M_{\lambda,\nu}(C)$-eigenvalues and heights for multiplicity free representations}

	In this section, we will give  detailed  calculations to show that  the eigenvalues of $M_{\lambda,\nu}(C)$ are different on each summand $ V_{\nu+\lambda_{i}}$ if and only if  $\lambda$ is  strongly multiplicity free.
	We will also prove that the equation $2 \text{ht}(\lambda) + 1 = \dim V_{\lambda}$ holds if and only if 
	$V_{\lambda}$ is  strongly multiplicity free through case by case study.
	 
	We have  known from Proposition \ref{smf_has_distinct_eigs} that  the eigenvalues of $M_{\lambda,\nu}(C)$ on each irreducible component $V_{\nu + \lambda_{i}} $ are distinct from each other if $\mathfrak{g}$-module $V_{\lambda}$ is strongly multiplicity free.
	For multiplicity free but not strongly multiplicity free $\mathfrak{g}$-module $V_{\lambda}$, we can construct certain distinct irreducible components occurring in some tensor product decomposition $V_{\lambda} \otimes V_{\nu}$ so that they have the same eigenvalues.

	\subsection{$A_{l} (l \geq 1)$}

	The fundamental modules $V_{\omega_{k}}$ ($k=1,\cdots, l$) are irreducible minuscule $A_{l}$-modules.
	Denote the weights of $V_{\omega_{k}}$ by $\lambda_{i_{1} \cdots i_{k}}$ with $1 \leq i_{1} < \dots < i_{k} \leq l+1 $.
	For $\nu =\sum\limits_{i=1}^{l}a_{i}\omega_{i} \in \mathcal{P}_{+}$,
	let the Young tableau of $\nu$ be of the form $[f(\nu)] = [f_{1}(\nu), \dots, f_{l}(\nu), $ $f_{l+1}(\nu)]$, where 
	$$f_{i}(\nu) = a_{i} + a_{i+1} + \dots + a_{l}$$ 
	for $i = 1, \dots, l$ and $ f_{1}(\nu) \geq \dots \geq f_{l}(\nu) \geq f_{l+1}(\nu) =0$. And we also have
	$$V_{\omega_{k}} \otimes V_{\nu} = \sum_{\substack{1 \leq i_{1} < \dots < i_{k} \leq l+1  \\ \nu + \lambda_{i_{1} \cdots i_{k}} \in \mathcal{P}_{+}}} V_{\nu +  \lambda_{i_{1} \cdots i_{k}}}. $$

	\begin{prop}
		Assume $V_{\nu +  \lambda_{i_{1} \dots i_{k}}}$ and $V_{\nu +  \lambda_{j_{1} \dots j_{k}}}$ are two different irreducible components occurring in the decomposition of  $V_{\omega_{k}} \otimes V_{\nu}$ with $i_{1} \leq j_{1}$. Then $\chi_{\nu + \lambda_{i_{1} \cdots i_{k}}}(C) = \chi_{\nu + \lambda_{j_{1} \cdots j_{k}}}(C)$ if and only if 
		$$
		\sum_{t = 1}^{k} (f_{i_{t}}(\nu) - f_{j_{t}}(\nu)) =  \sum_{t = 1}^{k}( i_{t} - j_{t} ).
		$$
		
	\end{prop}
	\begin{proof}
		Let $[f(\nu + \lambda_{i_{1} \cdots i_{k}})] = [f_{1}(\nu + \lambda_{i_{1} \cdots i_{k}}), \cdots, f_{l}(\nu + \lambda_{i_{1} \cdots i_{k}}), f_{l+1}(\nu + \lambda_{i_{1} \cdots i_{k}})] $ be the Young tableau of dominant weight $\nu + \lambda_{i_{1} \cdots i_{k}}$. 
		Then 
		\begin{equation*} 
			f_{t}(\nu + \lambda_{i_{1} \cdots i_{k}}) = 
			\begin{cases}
				0 & \textrm{ if } t = l + 1, \\
				f_{t}(\nu) + 1 & \textrm{ if } t \in \{ i_{1}, \dots, i_{k}\}, \\
				f_{t}(\nu) & \textrm{ otherwise}.
			\end{cases}
		\end{equation*}
			Since
		$$ \dfrac{\sum_{t = 1}^{l+1}f_{t}(\nu + \lambda_{i_{1} \cdots i_{k}})}{l+1} 
		= \dfrac{\sum_{t = 1}^{l+1}f_{t}(\nu + \lambda_{j_{1} \cdots j_{k}})}{l+1},$$
		we denote the above constant by $a_{0}$ and set
		$$c_{s} = a_{0} - (l+1) +s, s = 1, 2, \dots, l+1.$$ 
		By \eqref{C_{2} of A_{l}} we have
		$$\chi_{\nu + \lambda_{i_{1} \cdots i_{k}}}(C) - \chi_{\nu + \lambda_{j_{1} \cdots j_{k}}}(C)$$
		$$= \sum_{s = 1}^{l+1} (f_{s}(\nu + \lambda_{i_{1} \cdots i_{k}}) - c_{s})^{2} - \sum_{s = 1}^{l+1} (f_{s}(\nu + \lambda_{j_{1} \cdots j_{k}}) - c_{s})^{2}$$
		$$= (\sum_{t = 1}^{k}(f_{i_{t}}(\nu) + 1 - c_{i_{t}})^2 + \sum_{t = 1}^{k}(f_{j_{t}}(\nu) - c_{j_{t}})^2 - \sum_{t \in \{i_{1}, \dots, i_{k}\} \cap \{j_{1}, \dots, j_{k}\}}(f_{t}(\nu) - c_{t})^2) $$
		$$ \quad - (\sum_{t = 1}^{k} (f_{j_{t}}(\nu) + 1 - c_{j_{t}})^2 + \sum_{t = 1}^{k}(f_{i_{t}}(\nu) - c_{i_{t}})^2 - \sum_{t \in \{i_{1}, \dots, i_{k}\} \cap \{j_{1}, \dots, j_{k}\}}(f_{t}(\nu) - c_{t})^2)$$
		$$= [\sum_{t = 1}^{k} (f_{i_{t}}(\nu) + 1 - c_{i_{t}})^2  - \sum_{t = 1}^{k} (f_{i_{t}}(\nu) - c_{i_{t}})^2] - [\sum_{t = 1}^{k} (f_{j_{t}}(\nu) + 1 - c_{j_{t}})^2  - \sum_{t = 1}^{k} (f_{j_{t}}(\nu) - c_{j_{t}})^2]$$
		$$= 2\sum_{t = 1}^{k} (f_{i_{t}}(\nu) - f_{j_{t}}(\nu)) - 2\sum_{t = 1}^{k}(i_{t} - j_{t}).
		$$
These equations imply  the statement of the proposition.
	\end{proof}

	\begin{lemma}
		For $\nu =\sum\limits_{i=1}^{l} a_{i}\omega_{i} \in \mathcal{P}_{+}$, we have: 
		$$
		V_{k\omega_{1}} \otimes V_{\nu} = \bigoplus_{\substack{c_{1}, \dots, c_{l+1} \in \mathbb{N} \\ \sum_{i = 1}^{l+1} c_{i} = k \\ c_{i+1} \leq a_{i}, i = 1, \dots, l}} V_{b_{1}\omega_{1} + \dots + b_{l}\omega_{l}},
		$$
		where $b_{i} = a_{i} + c_{i} - c_{i+1}$ for $i = 1, 2, \dots, l$ (cf.  [FH]). 
		
	\end{lemma}

	\begin{prop}
		There exists some highest weight $\nu=\sum\limits_{i=1}^{l} a_{i}\omega_{i} \in \mathcal{P}_{+}$ such that the tensor product $V_{k\omega_{1}} \otimes V_{\nu}$ has two different irreducible $A_{l}$-components whose eigenvalues of $M_{k\omega_{1},\nu}(C)$ are the same if and only if $l \geq 2$.
	\end{prop}
	
	\begin{proof}
		Let $V_{\mu}$ be an arbitrary irreducible component occurring in the tensor decomposition $V_{k\omega_{1}} \otimes V_{\nu}$. 
		Write the  highest weight $\mu =\sum\limits_{i=1}^{l}b_{i}\omega_{i}$ with $b_{i} = a_{i} + c_{i} - c_{i+1}$, where $c_{i}$ satisfy the conditions in Lemma 5.2. 
		Let $[f(\mu)] = [f_{1}(\mu), \dots, f_{l+1}(\mu)]$ be the corresponding Young tableau of $\mu$, i. e. 
		$$f_{i}(\mu) = b_{i} + b_{i+1} + \dots + b_{l} = f_{i}(\nu) + c_{i} - c_{l+1} \   \text{for } i = 1, \dots, l,$$	
		$$f_{l+1}(\mu) = 0. $$
		Denote
		$$N = f_{1}(\nu) + \dots + f_{l+1}(\nu), \quad b_{0} = \frac{1}{l+1}(f_{1}(\mu) + \dots + f_{l+1}(\mu)),$$ 
		and set  
		$$\zeta_{i} = f_{i}(\mu) - b_{0} = f_{i}(\nu) + c_{i} - \frac{1}{l+1}(N+k)$$ 
		for $i = 1, \dots, l+1$. 
		Assume $V_{\mu'} =\sum\limits_{i=1}^{l}b_{i}' \omega_{i} $ is another irreducible component in  $V_{k\omega_{1}} \otimes V_{\nu}$ with $b_{i}' = a_{i} + c_{i}' - c_{i+1}'$, where $c_{i}'$ also satisfy the conditions in Lemma 5.2. 
		Then $\mu' \ne \mu$ since $k\omega_{1}$ is multiplicity free for $A_{l}$.
		Set 
		$$\zeta'_{i} = f_{i}(\nu) + c_{i}' - \frac{1}{l+1}(N+k)$$ 
		for $ i = 1, \dots, l+1$. 
		If the eigenvalues $\chi_{\mu}(C) = \chi_{\mu'}(C) $, by \eqref{C_{2} of A_{l}} we have 
		$$
		0 = \chi_{\mu}(C) - \chi_{\mu'}(C) = \sum_{i = 1}^{l+1} [(\zeta_{i} + (l+1) - i)^{2} - (\zeta'_{i} + (l+1) - i)^{2}] $$
		$$= \sum_{i = 1}^{l+1} [(f_{i}(\nu) + c_{i} - i)^{2} - (f_{i}(\nu) + c_{i}' - i)^{2}].
		$$
		Notice that
		$$
		\sum_{i = 1}^{l+1} (f_{i}(\nu) + c_{i} - i) = \sum_{i = 1}^{l+1} (f_{i}(\nu) + c_{i}' - i) = N + k - \frac{1}{2}(l+1)(l+2)
		$$
		is a constant and it is  not related to $c_{i}$ and $c_{i}'$.
		We write 
		$$x_{i} = f_{i}(\nu) + c_{i} - i, \quad y_{i} = f_{i}(\nu) + c_{i}' - i$$ 
		for $i = 1, \dots, l+1$.
		They satisfy the PTE problem
		\begin{equation*}
			\begin{cases}
					x_{1} + x_{2} + \dots + x_{l+1} = y_{1} + y_{2} + \dots + y_{l+1} \\
					x_{1}^{2} + x_{2}^2 + \dots + x_{l+1}^{2}= y_{1}^{2} + y_{2}^{2} + \dots + y_{l+1}^{2}
				\end{cases}.
		\end{equation*}
		Since the PTE problem with degree $2$ and size $l+1$ has a non-trivial integer solution $[x_{1}, x_{2}, \dots, x_{l+1}]$ and $[y_{1}, y_{2}, \dots, y_{l+1}]$, it follows that $l \geq 2$ by Lemma \ref{PTE_exist}.

		\medskip

		Now assume $l \geq 2$.
		For any fixed integer $k (k \geq 2)$, we define
		$$x_{1} = k+1, \quad x_{2} = k-3, \quad x_{3} = k-4,$$
		$$y_{1} = k, \quad y_{2} = k-1, \quad y_{3} = k-5,$$
		$$x_{i} = y_{i} = -i, \quad i = 4, \dots, l+1.$$
		Then, $[x_{1}, x_{2}, \dots, x_{l+1}]$ and $[y_{1}, y_{2}, \dots, y_{l+1}]$ is a non-trivial integral solution of the PTE problem with degree $2$ and size $l+1$  by Lemma \ref{solution_of_degree_2}.
		Set
		$$f_{1} = y_{2} + 2 = k+1, $$
		$$f_{2} = x_{3} + 3 = k-1, $$
		$$f_{i} = x_{i+1} + (i+1) = 0, \quad i = 3, \dots, l,$$
		and define
		$$c_{i} = x_{i} + i - f_{i}, \quad c_{i}' = y_{i} + i - f_{i}, \quad i = 1, \dots, l+1.$$
		Take $b_{i} = a_{i} + c_{i} - c_{i+1}$ and $b_{i}' = a_{i} + c_{i}' - c_{i+1}'$ for $i = 1, \dots, l$.
		It is easy to check that $c_{i}$ and $c_{i}'$ satisfy the conditions of the direct sum decomposition in Lemma 5.2.  
		Therefore, $V_{b_{1}\omega_{1} + \dots + b_{l}\omega_{l}}$ and $V_{b_{1}' \omega_{1} + \dots + b_{l}' \omega_{l}}$ are distinct irreducible components occurring in $V_{k\omega_{1}} \otimes V_{\nu}$ and they have the same $M_{k\omega_{1},\nu}(C)$-eigenvalues.
	\end{proof}
	
	Assume $V_{\lambda}$ is an irreducible $\mathfrak{sl}_{2}$-module of dimension great than $1$. For any  $\nu \in \mathcal{P}_{+}$, the eigenvalue of $M_{k\omega_{1},\nu}(C)$ acting on each irreducible $\mathfrak{sl}_{2}$-component of $V_{\lambda} \otimes V_{\nu}$ is distinct according to Proposition \ref{smf_has_distinct_eigs}.
	
	The similar result for $\lambda = k\omega_{l} (k \geq 2)$ can be proved by the same method as above, since $V_{k\omega_{l}}$ is dual to $V_{k\omega_{1}}$.

	\begin{prop} The equation   $2 \text{ht}(\lambda) + 1 = \dim V_{\lambda}$ holds for the multiplicity free $A_{l}$-module $V_{\lambda}$ if and only if $\lambda = \omega_{1}, \omega_{l}$ for $A_{l}$ and $k\omega_{1}$ for $A_{1}$.
	\end{prop}
	
	\begin{proof}

		By Weyl's dimension formula, we have $\dim V_{\omega_{k}} = \binom{l+1}{k}$ and $\dim V_{k\omega_{1}} = \dim V_{k \omega_{l}} = \binom{k+l}{l}$. 
		Since
		\[\omega_{k} = \frac{1}{l+1} [(l-k+1)\alpha_{1} + 2(l-k+1)\alpha_{2} + \dots + (k-1)(l-k+1)\alpha_{k-1}\]
		\[+ k(l-k+1)\alpha_{k} + k(l-k)\alpha_{k+1} + \dots + k\cdot 2 \alpha_{l-1} + k\alpha_{l}],\]
		we have $\text{ht}(\omega_{k}) = \frac{1}{2} k(l-k+1)$ and $\text{ht}(k\omega_{1}) = \text{ht}(k\omega_{1}) = \frac{1}{2}kl$.
		
		\medskip
		
		For $V_{\omega_{k}}$, consider the polynomial function 
		$$F(k) = k(l-k+1) + 1.$$ 
		If $k = 1$ or $l$, then $F(1) = F(l) = l+1$.
		It follows that $\omega_{1}$ and $\omega_{l}$ satisfy the equation $2 \text{ht}(\lambda) + 1 = \dim V_{\lambda}$  for any $l \geq 1$.
		For $2 \leq k \leq l-1$  and $l \geq 3$, we have 
		$$F(k) = -(k - \frac{l+1}{2})^{2} + \frac{(l+1)^{2}}{4} + 1 \leq  \frac{(l+1)^{2}}{4} + 1,$$
		$$\binom{l+1}{k} \geq \binom{l+1}{2} = \frac{(l+1)l}{2},$$
		$$F(k) - \binom{l+1}{2} \leq \frac{(l+1)^{2}}{4} + 1 -  \frac{(l+1)l}{2} = -\frac{l^{2} - 5}{4}.$$
		It is obvious that $F(k) < \binom{l+1}{k}$ for $l \geq 3$. 
		Hence,  $2 \text{ht}(\lambda) + 1 \neq \dim V_{\lambda}$  for  $\lambda=\omega_{2}, \dots, \omega_{l-1}$, $l \geq 3$.
		
		\medskip
		
		For $V_{k\omega_{1}}$ or $V_{k\omega_{l}} (k \geq 2)$, consider the polynomial function 
		$$G(x) = kx + 1 - \frac{(x+k) \cdots (x+2)(x+1)}{k!}.$$
		For $x \geq 0$, the derivative at $x$ satisfies
		$$G^{'}(x) = k - \frac{1}{k!}\sum_{i = 1}^{k} \frac{\prod_{j = 1}^{k}(x+j)}{x+i} < k - \frac{1}{k!}k\frac{(x+k) \cdots (x+1)}{x+k} = k - \binom{x+k-1}{k-1},$$
		and
		$$G^{''}(x) < 0.$$
		It is obvious that $G^{'}(x)$ is monotonically decreasing. 
		For $x \geq 1$,
		$$G^{'}(x) \leq G^{'}(1) < k - \binom{k}{k-1} = k - \binom{k}{1} = 0.$$
		It follows that $G(x)$ is monotonically decreasing for $x \geq 1$. 
		Thus we have
		$$G(1) = k+1 - \binom{k+1}{k} = 0,$$
		$$G(l) \leq G(2) = 2k+1 - \binom{k+2}{k} = -\frac{1}{2}k(k-1) < 0.$$
		The proof is completed.
	\end{proof}

	\subsection{$B_{l} (l \geq 2)$}
	
	Denote the  weight  $$\mu_{i}(\text{diag}(0, a_{1}, \dots, a_{l}, -a_{1}, \dots, -a_{l})) = a_{i}, \ i = 1, \dots, l.$$
	We know that $\Pi(V_{\omega_{1}}) = \{ 0, \mu_{1}, \dots, \mu_{l}, -\mu_{1}, \dots, -\mu_{l} \}$. And the spin module $V_{\omega_{l}}$ is an irreducible minuscule $B_{l}$-module. Its weights are of the form $\frac{1}{2}\sum_{i = 1}^{l} \varepsilon_{i}\mu_{i}$ with $\varepsilon_{i} = \pm 1$ for $i = 1, \dots, l$. 
	Let $\nu = a_{1}\omega_{1} + \dots + a_{l}\omega_{l} \in \mathcal{P}_{+}$ and define
	$$\tilde{f}_{l} = \frac{1}{2}a_{l}, \quad \tilde{f}_{s} = \sum_{t = s}^{l-1} a_{t} + \frac{1}{2}a_{l}, \quad s = 1, \dots, l-1.$$

	\begin{prop}
		The $M_{\omega_{l}, \nu}(C)$-eigenvalues of all irreducible components of $V_{\omega_{l}} \otimes V_{\nu}$ are distinct for any $\nu \in \mathcal{P}_{+}$ if and only if $l = 2$.
	\end{prop}
	
	\begin{proof}
		Let $\varepsilon = \frac{1}{2}\sum_{i = 1}^{l} \varepsilon_{i}\mu_{i}$ and $\eta = \frac{1}{2}\sum_{i = 1}^{l} \eta_{i}\mu_{i}$ be two distinct weights of $V_{\omega_{l}}$. 
		From \eqref{B_{l}} we obtain that the coordinates $\zeta_{i}$ corresponding to the irreducible component $V_{\nu + \varepsilon}$ in $V_{\omega_{l}} \otimes V_{\nu}$ are of the form $\tilde{f}_{i} + \frac{1}{2} \varepsilon_{i}$ for $i = 1, \dots, l$.
			By \eqref{C_{2} of B_{l}} we have
		$$\chi_{\nu + \varepsilon}(C) - \chi_{\nu + \eta}(C)$$
		$$= \sum_{i = 1}^{l} [(\tilde{f}_{i} + \frac{1}{2}\varepsilon_{i} - i + 2l)^{2} - (\tilde{f}_{i} + \frac{1}{2}\eta_{i} - i + 2l)^{2}] + [(\tilde{f}_{i} + \frac{1}{2}\varepsilon_{i} - i + 1)^{2} - (\tilde{f}_{i} + \frac{1}{2}\eta_{i} - i + 1)^{2}] $$
		$$= \sum_{i = 1}^{l} (\varepsilon_{i} - \eta_{i})(2\tilde{f}_{i} + 2(l-i) + 1). $$
		Then $\varepsilon_{i} - \eta_{i} = 0$ or $2, -2$.
		
		\medskip
		
		If $l = 2$, the above expression  is equal to $0$ if and only if
		$$(\varepsilon_{1} - \eta_{1})(2\tilde{f}_{1} + 3) + (\varepsilon_{2} - \eta_{2})(2\tilde{f}_{2} + 1) = 0.$$
		But it is never equal to $0$ for any $a_{1}, a_{2} \in \mathbb{N}$. 
		
		\medskip
		
		If $l \geq 3$, let $\varepsilon_{i} = \eta_{i}$ for $i = 1, \dots, l-3$. Then the above equation  is equal to $0$ if and only if
		$$(\varepsilon_{l-2} - \eta_{l-2})(2\tilde{f}_{l-2} + 5) + (\varepsilon_{l-1} - \eta_{l-1})(2\tilde{f}_{l-1} + 3) + (\varepsilon_{l} - \eta_{l})(2\tilde{f}_{l} + 1) = 0.$$
		Set 
		$$\varepsilon_{l-2} - \eta_{l-2} = 2, \quad \varepsilon_{l-1} - \eta_{l-1} = \varepsilon_{l} - \eta_{l} = -2.$$ 
		Then,
		$$2a_{l-2} + 1 = a_{l}.$$
		If $a_{i}, i = 1, \dots, l$ satisfy $2a_{l-2} + 1 = a_{l}$ and are sufficiently large such that $\nu + \varepsilon, \nu + \eta \in \mathcal{P}_{+}$, then irreducible $B_{l}$-modules $V_{\nu + \varepsilon}$ and $V_{\nu + \eta}$ occurring in $V_{\omega_{l}} \otimes V_{\nu}$ have the same $M_{\omega_{l}, \nu}(C)$-eigenvalue for $l \geq 3$.

	\end{proof}

	\begin{prop}
		For the  multiplicity free $B_{l}$-module $V_{\lambda}$, the equation 
		 $2 \text{ht}(\lambda) + 1 = \dim V_{\lambda}$ holds if and only if $\lambda = \omega_{1}$.
	\end{prop}
	
	\begin{proof}

		By Weyl's dimension formula, we have $\dim V_{\omega_{1}} = 2l+1$ and $\dim V_{\omega_{l}} = 2^{l}$. 
	Using the inverse of the Cartan matrix, we know that 
		$$\omega_{1} = \alpha_{1} + \dots + \alpha_{l},$$ 
		$$\omega_{l} = \frac{1}{2}(\alpha_{1} + 2\alpha_{2} + \dots + l\alpha_{l}).$$ 
		Then $\text{ht}(\omega_{1}) = l$ and $\text{ht}(\omega_{l}) = \frac{(l+1)l}{4}$. 
		
		\medskip
		
		For $V_{\omega_{1}}$, we have $2\text{ht}(\omega_{1}) + 1 = 2l+1 = \dim V_{\omega_{1}}$. 
		
		\medskip
		
		For $V_{\omega_{l}}$, consider the polynomial function 
		$$F_{B}(x) = \frac{1}{2}x^{2} + \frac{1}{2}x + 1 - 2^{x}$$
		and its derivatives are of the form
		$$F_{B}^{'}(x) = x + \frac{1}{2} - 2^{x} \text{ln}2,$$
		$$F_{B}^{''}(x) = -2^{x}(\text{ln}2)^{2} < 0.$$
		Hence $F_{B}^{'}(x)$ is monotonically decreasing. 
		For $x \geq 2$, we have
		$$F_{B}^{'}(x) \leq F_{B}^{'}(2) = \frac{5}{2} - 4\text{ln}2 < 0,$$
		thus $F_{B}(x)$ is also monotonically decreasing for $x \geq 2$. 
		Then we know
		$$F_{B}(l) < F_{B}(2) = 0$$
		for $l \geq 3$. 
		Only the case $l = 2$ satisfies the equation $2\text{ht}(\lambda) + 1 = \dim V_{\lambda}$.
		The proof is completed.
	\end{proof}

	\subsection{$C_{l} (l \geq 3)$}

    Denote the weight   
    $$\mu_{i}(\text{diag}(a_{1}, a_{2}, a_{3}, -a_{1}, -a_{2}, -a_{3}))= a_{i}, \ i = 1, 2, 3.$$
	The weight set for $C_{3}$-module $V_{\omega_{3}}$ is 
	$$\Pi(V_{\omega_{3}}) = \{ \pm \mu_{1},  \pm \mu_{2},  \pm \mu_{3},  \pm (\mu_{1} + \mu_{2} + \mu_{3}), \pm (\mu_{1} + \mu_{2} - \mu_{3}), $$ 
	$$\pm (\mu_{1} - \mu_{2} + \mu_{3}), \pm (\mu_{1} - \mu_{2} - \mu_{3})\}.$$ 
	According to Theorem 4.6 in [VD], we have the following proposition: 
	
	\begin{prop}
		For  $\nu = \sum\limits_{i=1}^{3}a_{i}\omega_{i} \in \mathcal{P}_{+}$, the decomposition of $C_{3}$-module $V_{\omega_{3}} \otimes V_{\nu}$ is of the form
		$$
		V_{\omega_{3}} \otimes V_{\nu} = \sum_{\substack{J \subseteq \{1, 2, 3\} \\ |J| = 1, 3}} \sum_{\substack{\varepsilon_{j} \in \{1, -1\} \\ j \in J \\ \nu + e_{\varepsilon J} \in \mathcal{P}_{+}}} \biggl( \sum_{\substack{|J| \leq s \leq 3 \\ 3-s \textrm{ is even}}} C_{3, s}^{(3)} \cdot N_{1^{s}, e_{\varepsilon J}}(\rho + \nu) \biggr) V_{\nu + e_{\varepsilon J}},
		$$
		where
		$$
		e_{\varepsilon J} = \sum_{j \in J} \varepsilon_{i} \mu_{i}, \quad J \subseteq \{1, 2, 3\}, \quad \varepsilon_{1}, \varepsilon_{2}, \varepsilon_{3} \in \{1, -1\},
		$$
		and the coefficients above are given by
		$$
		N_{1^{s}, e_{\varepsilon J}}(\rho + \nu) = 2^{s - |J|} \binom{3 - |J|}{s - |J|} - \# \{e_{\varepsilon I} | I \subseteq \{1, 2, 3\} \setminus J,
		$$
		$$ |I| = s - |J|, \ \text{ and } \exists \ i, i+1 \in I, \text{ s.t. } a_{i} = 0, -\varepsilon_{i} = \varepsilon_{i+1} = 1\},$$
		$$
		C_{3, s}^{(3)} = 4^{3 - s}\dfrac{(1)_{3-s} (\frac{3}{2})_{3-s}}{(3-s)! (3)_{3-s}} - \sum_{\substack{1 \leq k \leq 2-s \\ k \text{ odd}}} 4^{k} \dfrac{(1)_{k} (\frac{3}{2})_{k}}{k! (3)_{k}} K_{3-k, 3-s-k}^{(3)},
		$$
		where $(z)_{k} = z(z+1)(z+2) \cdots (z+k-1)$ for $z \in \mathbb{R}$ and $k \in \mathbb{N}^{*}$(with the convention that $(z)_{0} = 1$), and
		$$
		K_{r, l}^{(3)} = (-1)^{\frac{1}{2} (l-1)} 2^{l} \binom{3-r+l}{l}T_{l}
		$$
		for $1 \leq l \leq r$ with $l$ odd, and
		$$
		T_{2k+1} = \sum_{1 \leq s \leq 2k+1} (-1)^{k+s+1} 2^{2k+1-s} \bigg( \sum_{1 \leq t \leq s} (-1)^{s-t} \binom{s}{t} t^{2k+1} \bigg).
		$$
	\end{prop}
	
	\begin{cor}\label{mult_C_{3}_omega3}
		If $\nu + \mu_{2}, \nu + (\mu_{1} - \mu_{2} + \mu_{3}) \in \mathcal{P}_{+}$, then the multiplicities of them occurring in $V_{\omega_{3}} \otimes V_{\nu}$ are equal to $1$.
	\end{cor}
	
	\begin{proof}
		Let $n_{\nu + \mu_{2}}$ be the multiplicity of $\nu + \mu_{2}$ and $n_{\nu + (\mu_{1} - \mu_{2} + \mu_{3})}$ be the multiplicity of $ \nu + (\mu_{1} - \mu_{2} + \mu_{3})$. Then
		$$
		n_{\nu + \mu_{2}} = C_{3, 1}^{(3)} N_{1^{1}, e_{(0, 1, 0)}} (\rho + \nu) + C_{3,3}^{(3)} N_{1^{3}, e_{(0, 1, 0)}} (\rho + \nu) = -3 \cdot 1 + 1 \cdot 4 = 1, $$
		and
		$$n_{\nu + (\mu_{1} - \mu_{2} + \mu_{3})} = C_{3,3}^{(3)} N_{1^{3}, e_{(1, -1, 1)}} (\rho + \nu) = 1 \cdot 1 = 1.
		$$
	\end{proof}
	
	\begin{prop}
		There exists some weight $\nu \in \mathcal{P}_{+}$ so that $\nu + \mu_{2}, \nu + (\mu_{1} - \mu_{2} + \mu_{3}) \in \mathcal{P}_{+}$ and irreducible $C_{3}$-components $V_{\nu + \mu_{2}}$ and $V_{\nu + (\mu_{1} - \mu_{2} + \mu_{3})}$ of $V_{\omega_{3}} \otimes V_{\nu}$ have the same $M_{\omega_{3}, \nu}(C)$-eigenvalues.
	\end{prop}
	
	\begin{proof}
		Let $\nu = a_{1}\omega_{1} + a_{2}\omega_{2} + a_{3}\omega_{3} \in \mathcal{P}_{+}$.
		By \eqref{C_{2} of C_{l}} we have that $\chi_{\nu + \mu_{2}}(C) = \chi_{\nu + (\mu_{1} - \mu_{2} + \mu_{3})}(C)$ if and only if $a_{2} = a_{1} + 1$. 
		In other words, if $a_{i}, i = 1, 2, 3$ satisfy $a_{2} = a_{1} + 1$ and are sufficiently large such that $\nu + \mu_{2}, \nu + (\mu_{1} - \mu_{2} + \mu_{3}) \in \mathcal{P}_{+}$, irreducible $C_{3}$-modules $V_{\nu + \mu_{2}}$ and $V_{\nu + (\mu_{1} - \mu_{2} + \mu_{3})}$ occurs in the decomposition of  $V_{\omega_{3}} \otimes V_{\nu}$ with multiplicity $1$ by Corollary \ref{mult_C_{3}_omega3} and have the same $M_{\omega_{3}, \nu}(C)$-eigenvalue.
	\end{proof}

	\begin{prop}
		For the multiplicity free $C_{l}$-module $V_{\lambda}$, the equation 
		 $2 \text{ht}(\lambda) + 1 = \dim V_{\lambda}$ holds  if and only if $\lambda = \omega_{1}$.
	\end{prop}
	
	\begin{proof}For the irreducible $C_{l}$-modules $V_{\omega_{1}}$ and $C_{3}$-module $V_{\omega_{3}}$, we have $\dim V_{\omega_{1}} = 2l$ and $\dim V_{\omega_{3}} = 14$ by Weyl's dimension formula. Using the inverse of the Cartan matrix of $C_{l}$, we get $$\omega_{1} = \alpha_{1} + \dots + \alpha_{l-1} + \frac{1}{2} \alpha_{l},$$
		$$\omega_{l} = \alpha_{1} + \dots + (l-1)\alpha_{l-1} + \frac{l}{2} \alpha_{l}.$$ 
		Then, $\text{ht}(\omega_{1}) = l - \frac{1}{2}$ for $C_{l}$-modules $V_{\omega_{1}}$ and $\text{ht}(\omega_{3}) = \frac{9}{2}$ for $C_{3}$-module $V_{\omega_{3}}$. Furthermore, for  $C_{l}$-modules $V_{\omega_{1}}$, we have $2\text{ht}(\omega_{1}) + 1 = 2l = \text{dim} V_{\omega_{1}}$. And, for $C_{3}$-module $V_{\omega_{3}}$, we have $2\text{ht}(\omega_{3}) + 1 = 10 < 14 = \dim V_{\omega_{3}}$. 
		The proof is completed.
	\end{proof}

	\subsection{$D_{l} (l \geq 4)$}
	
	Denote weight 
	$$\mu_{i}(\text{diag}(a_{1}, \dots, a_{l}, -a_{1}, \dots, -a_{l}))= a_{i}, \ i = 1, \dots, l.$$
	We know that $$\Pi(V_{\omega_{1}}) = \{ \mu_{1}, \dots, \mu_{l}, -\mu_{1}, \dots, -\mu_{l} \}.$$ 
	It is well known that the irreducible $D_{l}$-modules $V_{\omega_{l-1}}$ and $V_{\omega_{l}}$ are spin modules with dimension $2^{l-1}$. 
	If $l$ is even,  the weights of $V_{\omega_{l}}$ are of the form $\frac{1}{2} \sum_{i = 1}^{l} \varepsilon_{i} \mu_{i}$ with $\varepsilon_{i} \in \{ 1, -1 \} \text{ for } i = 1, \dots, l$, which have an even number of $\varepsilon_{i}$ negative and those of $V_{\omega_{l-1}}$ are of the same form which have an odd number negative. 
	If $l$ is odd, we have the reverse situation. 
	Besides, $V_{\omega_{1}}, V_{\omega_{l-1}}$ and $V_{\omega_{l}}$ are minuscule $D_{l}$-modules.
	
	Write $\nu = a_{1}\omega_{1} + \dots + a_{l}\omega_{l} \in \mathcal{P}_{+}$. Denote
	$$
	\tilde{f}_{i} = a_{i} + a_{i+1} + \dots + a_{l-2} + \frac{1}{2}(a_{l-1} + a_{l}), \quad i = 1,2, \dots, l-2, $$
	$$
	\tilde{f}_{l-1} = \frac{1}{2}(a_{l-1} + a_{l}), \quad
	\tilde{f}_{l} = \varepsilon_{l, 2\mathbb{Z}}(\frac{1}{2} a_{l-1} - \frac{1}{2} a_{l}),
	$$
	where 
	\begin{equation*}
		\varepsilon_{l, 2\mathbb{Z}} = 
		\begin{cases}
			1 &\text{ if } l \notin 2\mathbb{Z} \\
			-1 &\text{ if } l \in 2\mathbb{Z}
		\end{cases}.
	\end{equation*}

	\begin{prop}
		Assume  $\nu =\sum\limits_{i=1}^{l} a_{i}\omega_{i} \in \mathcal{P}_{+}$ satisfies $a_{l-1} = a_{l}$ and $a_{i}, i = 1, \dots, l$ are sufficiently large so that $\nu + \mu_{l}$ and $\nu - \mu_{l} \in \mathcal{P}_{+}$. Then irreducible $D_{l}$-components $V_{\nu + \mu_{l}}$ and $V_{\nu - \mu_{l}}$ of $V_{\omega_{1}} \otimes V_{\nu}$ have the same $M_{\omega_{1}, \nu}(C)$-eigenvalues.
	\end{prop}
	
	\begin{proof}
		If $l$ is even, it follows from \eqref{C_{2} of D_{l}} that
		$$
		\chi_{\nu + \mu_{l}}(C) - \chi_{\nu - \mu_{l}}(C) $$
		$$= [(\tilde{f}_{l} + 1 - l + 2l - 1)^{2} - (\tilde{f}_{l} - l + 2l - 1)^{2}] + [(\tilde{f}_{l} + 1 - l + 1)^{2} - (\tilde{f}_{l} - l + 1)^{2}] $$
		$$
		\quad - [(\tilde{f}_{l} - 1 - l + 2l - 1)^{2} - (\tilde{f}_{l} - l + 2l - 1)^{2}] - [(\tilde{f}_{l} - 1 - l + 1)^{2} - (\tilde{f}_{l} - l + 1)^{2}] $$
		$$
		= 8\tilde{f}_{l} = 8(-\frac{1}{2} a_{l-1} + \frac{1}{2} a_{l}) = 0.
		$$
		
		\medskip
		
		If $l$ is odd, similarly we get
		$$
		\chi_{\nu + \mu_{l}}(C) - \chi_{\nu - \mu_{l}}(C) = -8\tilde{f}_{l} = -8(\frac{1}{2} a_{l-1} - \frac{1}{2} a_{l}) = 0.
		$$
		These equations imply the proposition.    
	\end{proof}

	\begin{prop}
		For $l \geq 4$, there exists some weight $\nu \in \mathcal{P}_{+}$ so that $V_{\omega_{l-1}} \otimes V_{\nu}$  (resp. $V_{\omega_{l}} \otimes V_{\nu}$) has two irreducible $D_{l}$-components with the same $M_{\omega_{l-1}, \nu}(C)$-eigenvalues (resp. $M_{\omega_{l}, \nu}(C)$-eigenvalues).
	\end{prop}

	\begin{proof}
		Assume $l (l \geq 4)$ is even.
		Denote $\varepsilon = \frac{1}{2} \sum_{i = 1}^{l} \varepsilon_{i} \mu_{i}$ and $\eta = \frac{1}{2} \sum_{i = 1}^{l} \eta_{i} \mu_{i}$ with undetermined coefficients $\varepsilon_{i}, \eta_{i} \in \{ 1, -1 \}$ for $i = 1, \dots, l$. 
		Take $\nu = a_{1}\omega_{1} + \dots + a_{l}\omega_{l} \in \mathcal{P}_{+}$ so that $\nu + \varepsilon, \nu + \eta \in \mathcal{P}_{+}$. 
		Using \eqref{C_{2} of D_{l}}, we get
		$$\chi_{\nu + \varepsilon}(C) - \chi_{\nu + \eta}(C) $$
		$$
		= \sum_{i = 1}^{l} [(\tilde{f}_{i} + \frac{1}{2}\varepsilon_{i} - i + 2l - 1)^{2} + (\tilde{f}_{i} + \frac{1}{2}\varepsilon_{i} - i + 1)^{2}] $$
		$$
		\quad\quad - \sum_{i = 1}^{l} [(\tilde{f}_{i} + \frac{1}{2}\eta_{i} - i + 2l - 1)^{2} + (\tilde{f}_{i} + \frac{1}{2}\eta_{i} - i + 1)^{2}] $$
		$$
		= \sum_{i = 1}^{l} 2(\varepsilon_{i} - \eta_{i})(\tilde{f}_{i} - i + l).$$ 
		
		\medskip
		
		Set 
		\begin{equation} \label{varepsilon_eta_odd}
			\begin{split}
				&\varepsilon_{i} = \eta_{i} = 1 \quad \text{if} \quad i = 1, \dots, l-4, \\
				&\varepsilon_{l-3} = 1, \quad \varepsilon_{l-2} = \varepsilon_{l-1} = \varepsilon_{l} = -1, \\
				&\eta_{l-3} = -1, \quad \eta_{l-2} = \eta_{l-1} = \eta_{l} = 1.
			\end{split}
		\end{equation}
		It is obvious that $\varepsilon, \eta \in \Pi(V_{\omega_{l-1}})$. 
		Therefore, $\chi_{\nu + \varepsilon}(C) = \chi_{\nu + \eta}(C)$ if and only if $a_{l-3} = a_{l}$. 
		In other words, if $a_{i}, i = 1, \dots, l$ satisfy $a_{l-3} = a_{l}$ and are sufficiently large such that $\nu + \varepsilon, \nu + \eta \in \mathcal{P}_{+}$, where $\varepsilon, \eta$ are defined in \eqref{varepsilon_eta_odd}, the irreducible $D_{l}$-modules $V_{\nu + \varepsilon}$ and $V_{\nu + \eta}$ occurring in $V_{\omega_{l-1}} \otimes V_{\nu}$ have the same $M_{\omega_{l-1}, \nu}(C)$-eigenvalue.
		
		\medskip
		
		Set 
		\begin{equation} \label{varepsilon_eta_even}
			\begin{split}
				\varepsilon_{i} = \eta_{i} = 1 \quad \text{if} \quad i = 1, \dots, l-4, \\
				\varepsilon_{l-3} = \varepsilon_{l} = 1, \quad \varepsilon_{l-2} = \varepsilon_{l-1} = -1, \\
				\eta_{l-3} = \eta_{l} = -1, \quad \eta_{l-2} = \eta_{l-1} = 1.
			\end{split}
		\end{equation}
		It is obvious that $\varepsilon, \eta \in \Pi(V_{\omega_{l}})$. 
		Then $\chi_{\nu + \varepsilon}(C) = \chi_{\nu + \eta}(C)$ if and only if $a_{l-3} = a_{l-1}$. 
		In other words, if $a_{i}, i = 1, \dots, l$ satisfy $a_{l-3} = a_{l-1}$ and are sufficiently large such that $\nu + \varepsilon, \nu + \eta \in \mathcal{P}_{+}$, where $\varepsilon, \eta$ are defined in \eqref{varepsilon_eta_even}, then irreducible $D_{l}$-modules $V_{\nu + \varepsilon}$ and $V_{\nu + \eta}$ occurring in $V_{\omega_{l}} \otimes V_{\nu}$ have the same $M_{\omega_{l}, \nu}(C)$-eigenvalue.

		\medskip
		
		Assume $l (l \geq 4)$ is odd.
		Similarly, we denote $\varepsilon = \frac{1}{2} \sum_{i = 1}^{l} \varepsilon_{i} \mu_{i}$ and $\eta = \frac{1}{2} \sum_{i = 1}^{l} \eta_{i} \mu_{i}$ with undetermined coefficients $\varepsilon_{i}, \eta_{i} \in \{ 1, -1 \}$ for $i = 1, \dots, l$ and take $\nu = a_{1}\omega_{1} + \dots + a_{l}\omega_{l} \in \mathcal{P}_{+}$ so that $\nu + \varepsilon, \nu + \eta \in \mathcal{P}_{+}$. 
		Using \eqref{C_{2} of D_{l}} we get
		$$\chi_{\nu + \varepsilon}(C) - \chi_{\nu + \eta}(C) $$
		$$
		= \sum_{i = 1}^{l-1} [(\tilde{f}_{i} + \frac{1}{2}\varepsilon_{i} - i + 2l - 1)^{2} + (\tilde{f}_{i} + \frac{1}{2}\varepsilon_{i} - i + 1)^{2}] 
		+ [(\tilde{f}_{l} - \frac{1}{2}\varepsilon_{l} + l - 1)^{2} + (\tilde{f}_{l} - \frac{1}{2}\varepsilon_{l} - l + 1)^{2}] $$
		$$
		- \sum_{i = 1}^{l-1} [(\tilde{f}_{i} + \frac{1}{2}\eta_{i} - i + 2l - 1)^{2} + (\tilde{f}_{i} + \frac{1}{2}\eta_{i} - i + 1)^{2}] 
		- [(\tilde{f}_{l} - \frac{1}{2}\eta_{l} + l - 1)^{2} + (\tilde{f}_{l} - \frac{1}{2}\eta_{l} - l + 1)^{2}]$$
		$$
		= \sum_{i = 1}^{l-1} 2(\varepsilon_{i} - \eta_{i})(\tilde{f}_{i} - i + l) - 2(\varepsilon_{l} - \eta_{l})\tilde{f}_{l}.$$ 
		It is obvious that $\varepsilon, \eta$ defined in \eqref{varepsilon_eta_odd} are in $\Pi(V_{\omega_{l}})$ and $\varepsilon, \eta$ defined in \eqref{varepsilon_eta_even} are in $\Pi(V_{\omega_{l-1}})$.
		Through similar discussions, we have irreducible $D_{l}$-modules $V_{\nu + \varepsilon}$ and $V_{\nu + \eta}$ with $\varepsilon, \eta$ defined in \eqref{varepsilon_eta_odd} occurring in  $V_{\omega_{l}} \otimes V_{\nu}$ (resp. those with $\varepsilon, \eta$ defined in \eqref{varepsilon_eta_even} occurring in $V_{\omega_{l-1}} \otimes V_{\nu}$) have the same $M_{\omega_{l}, \nu}(C)$-eigenvalues (resp. $M_{\omega_{l-1}, \nu}(C)$-eigenvalues).		
	\end{proof}

	\begin{prop}
		For every  multiplicity free $D_{l}$-module $V_{\lambda}$, we have $2 \text{ht}(\lambda) + 1 \neq  \dim V_{\lambda}$.
	\end{prop}
	
	\begin{proof}		
		By Weyl's dimension formula, we have $\dim V_{\omega_{1}} = 2l$ and $\dim V_{\omega_{l-1}}$ $= \dim V_{\omega_{l}} = 2^{l-1}$. 
		Using the inverse of the Cartan matrix of $D_{l}$, we know that 
		$$\omega_{1} = \alpha_{1} + \dots + \alpha_{l-2} + \frac{1}{2}(\alpha_{l-1} + \alpha_{l}),$$ $$\omega_{l-1} = \frac{1}{2}(\alpha_{1} + \dots + (l-2)\alpha_{l-2}) + \frac{l}{4}\alpha_{l-1} + \frac{l-2}{4}\alpha_{l},$$ 
		$$\omega_{l} = \frac{1}{2}(\alpha_{1} + \dots + (l-2)\alpha_{l-2}) + \frac{l-2}{4}\alpha_{l-1} + \frac{l}{4}\alpha_{l}.$$ 
		Then $\text{ht}(\omega_{1}) = l-1$ and $\text{ht}(\omega_{l-1}) = \text{ht}(\omega_{l}) = \frac{l(l-1)}{4}$.
		
		\medskip
		
		For $V_{\omega_{1}}$, it is obvious that $2\text{ht} (\omega_{1}) + 1 = 2l-1 < 2l = \dim V_{\omega_{1}}$.
		
		\medskip
		
		For $V_{\omega_{l-1}}$ and $V_{\omega_{l}}$, consider the polynomial function 
		$$F_{D}(x) = \frac{1}{2}x^{2} - \frac{1}{2}x + 1 - 2^{x-1}.$$ 
		It is easy to see that $F_{D}(x) = F_{B}(x-1)$, then it follows that $F_{D}(x)$ is monotonically decreasing if $x \geq 3$ and $F_{D}(l) \leq F_{D}(4) = F_{B}(3) < 0$. 
		Since $V_{\omega_{l-1}}$ and $V_{\omega_{l}}$ have the same dimension and $\text{ht}(\omega_{l-1}) = \text{ht}(\omega_{l})$, the equation $2\text{ht}(\lambda) + 1 = \dim V_{\lambda}$ does not hold.
	\end{proof}

	\subsection{$G_{2}$}
	
	It follows from Proposition \ref{smf_has_distinct_eigs} that the $M_{\omega_{1}, \nu}(C)$-eigenvalues of all the irreducible components in the decomposition $V_{\omega_{1}} \otimes V_{\nu}$ are distinct for any $\nu \in \mathcal{P}_{+}$. 
	The Cartan matrix yields 
	$$\omega_{1} = 2\alpha_{1} + \alpha_{2},$$
	$$\omega_{2} = 3\alpha_{1} + 2\alpha_{2}.$$ 
	Then $\text{ht}(\omega_{1}) = 3$. 
	One knows that $\Pi(V_{\omega_{1}}) = \{0, \pm \omega_{1}, \pm (-\omega_{1} + \omega_{2}), \pm (2\omega_{1} - \omega_{2})\}$ and $V_{\omega_{1}}$ has dimension $7$. 
	Then we have $2 \text{ht}(\omega_{1}) + 1 = \dim V_{\omega_{1}}$.

	\subsection{$E_{6}$}
	
	We know that $V_{\omega_{1}}$ and $V_{\omega_{6}}$ are irreducible minuscule $E_{6}$-modules.  
	They are dual and have the same dimension $27$. 
	The $E_{6}$-module decomposition of $V_{\omega_{1}} \otimes V_{\nu}$ is presented in [KW].

	\begin{prop}
		There exists some weight $\nu \in \mathcal{P}_{+}$ such that $V_{\omega_{1}} \otimes V_{\nu}$ has two irreducible $E_{6}$-components with the same $M_{\omega_{1}, \nu}(C)$-eigenvalues.
	\end{prop}
	
	\begin{proof}
		Write the weight $\nu = a_{1}\omega_{1} + \dots + a_{6}\omega_{6}$. 
		One knows that $\lambda_{1} = -\omega_{1} + \omega_{6}$ and $\lambda_{2} = \omega_{3} - \omega_{5} \in \Pi(V_{\omega_{1}})$. 
		Define $F(\mu) = \chi_{\nu + \mu}(C) -\chi_{\nu}(C)$ for each $\mu \in \Pi(V_{\omega_{1}})$. 
		Then we have
		$$
		F(\lambda_{1}) = \frac{2}{3}(-2a_{1} - a_{3} + a_{5} + 2a_{6}) + \frac{4}{3}, $$
		$$
		F(\lambda_{2}) = \frac{2}{3}(a_{1} + 2a_{3} - 2a_{5} - a_{6}) + \frac{4}{3}.
		$$
		Hence $F(\lambda_{1}) = F(\lambda_{2})$ if and only if $a_{1} + a_{3} = a_{5} + a_{6}$. 
		In other words, if the positive integers $a_{i} (i = 1, \dots, 6) $
		are sufficiently large such that $\nu + \lambda_{1}, \nu + \lambda_{2} \in \mathcal{P}_{+}$ and they		
		satisfy the equation		
	  $a_{1} + a_{3} = a_{5} + a_{6}$, then the  irreducible $E_{6}$-modules $V_{\nu + \lambda_{1}}$ and $V_{\nu + \lambda_{2}}$ occurring in $V_{\omega_{1}} \otimes V_{\nu}$ have the same $M_{\omega_{1}, \nu}(C)$-eigenvalues.
	\end{proof}
	
	The similar conclusion of $V_{\omega_{6}} \otimes V_{\nu}$ can be proved by the  above  method  since $V_{\omega_{1}}$ is the dual module to $V_{\omega_{6}}$.

	\begin{prop}
		For the multiplicity free $E_{6}$-modules $V_{\omega_{1}}$ and $V_{\omega_{6}}$, we have  $2 \text{ht}(\lambda) + 1 \neq  \dim V_{\lambda}$.
	\end{prop}
	
	\begin{proof}		
	The information about the Cartan matrix listed in [Hu] yields:
		$$\omega_{1} = \frac{1}{3}(4\alpha_{1} + 3\alpha_{2} + 5\alpha_{3} + 6\alpha_{4} + 4\alpha_{5} + 2\alpha_{6}),$$ 
		$$\omega_{6} = \frac{1}{3}(2\alpha_{1} + 3\alpha_{2} + 4\alpha_{3} + 6\alpha_{4} + 5\alpha_{5} + 4\alpha_{6}).$$ 
		Then $\text{ht}(\omega_{1}) = \text{ht}(\omega_{6}) = 8$. 
		Since $V_{\omega_{1}}$ and $V_{\omega_{6}}$ have the same dimension $27$, it follows that $2 \text{ht}(\lambda) + 1 < \dim V_{\lambda}$.
	\end{proof}

	\subsection{$E_{7}$}
	
	We know that $V_{\omega_{7}}$ is an irreducible minuscule $E_{7}$-module with dimension $56$.  
	The $E_{7}$-module decomposition of $V_{\omega_{7}} \otimes V_{\nu}$ is also presented in [KW].

	\begin{prop}
		There exists some weight $\nu \in \mathcal{P}_{+}$ such that $V_{\omega_{7}} \otimes V_{\nu}$ has two irreducible $E_{7}$-components with the same $M_{\omega_{7}, \nu}(C)$-eigenvalues.
	\end{prop}

	\begin{proof}
		Write $\nu = a_{1}\omega_{1} + \dots + a_{7}\omega_{7}$. 
		From [KW], one knows that $\lambda_{1} = \omega_{1} - \omega_{7}$ and $\lambda_{2} = - \omega_{2} + \omega_{6} \in \Pi(V_{\omega_{7}})$. 
		Define $F(\mu) = \chi_{\nu + \mu}(C) - \chi_{\nu}(C)$ for each $\mu \in \Pi(V_{\omega_{7}})$. Then we have
		$$
		F(\lambda_{1}) = 2a_{1} + a_{2} + 2a_{3} + 2a_{4} + a_{5} - a_{7} + (7 + \frac{3}{2}), $$
		$$
		F(\lambda_{2}) = -a_{2} + a_{5} + 2a_{6} + a_{7} + (3 + \frac{3}{2}).
		$$
		Hence $F(\lambda_{1}) = F(\lambda_{2})$ if and only if $a_{1} + a_{2} + a_{3} + a_{4} + 2 = a_{6} + a_{7}$. 
		In other words, if the positive integers $a_{i} (i = 1, \dots, 7) $
		 are sufficiently large such that $\nu + \lambda_{1}, \nu + \lambda_{2} \in \mathcal{P}_{+}$ and they		
		 satisfy the equation $a_{1} + a_{2} + a_{3} + a_{4} + 2 = a_{6} + a_{7}$, then the  irreducible $E_{7}$-modules $V_{\nu + \lambda_{1}}$ and $V_{\nu + \lambda_{2}}$ occurring in $V_{\omega_{7}} \otimes V_{\nu}$ have the same $M_{\omega_{7}, \nu}(C)$-eigenvalues.
	\end{proof}

	\begin{prop}
		For  $E_{7}$-module $V_{\omega_{7}}$,  we have $2\text{ht}(\lambda) + 1 \neq \dim V_{\lambda}$.
	\end{prop}
	
	\begin{proof}
		By means of  the Cartan matrix listed in [Hu], we have 
		$$\omega_{7} = \frac{1}{2}(2\alpha_{1} + 3\alpha_{2} + 4\alpha_{3} + 6\alpha_{4} + 5\alpha_{5} + 4\alpha_{6} + 3\alpha_{7}).$$ 
		Then $\text{ht}(\omega_{7}) = \frac{27}{2}$. 
		Since $ \dim  V_{\omega_{7}}=56$, we get $2 \text{ht}(\lambda) + 1 < \dim V_{\lambda}$.
	\end{proof}

	\appendix
	\section{Eigenvalues of second order Casimir operators for classical simple Lie algebras}

	The formulas of the eigenvalues of Casimir operators of arbitrary order for all classical Lie algebras are calculated in [Po2].
	In this appendix, we recall some formulas of the  eigenvalues of second order Casimir operators for classical Lie algebras.

	Let $\mathfrak{g}$ be the classical simple Lie algebra and $(\zeta) = (\zeta_{i})_{i \in I}$ be the highest weight of irreducible $\mathfrak{g}$-module $V_{\zeta}$, where coordinates $\zeta_{i}$ and the index set $I$ in the case of classical Lie algebras are determined explicitly by the corresponding  Young tableau of $\zeta$.
	Popov gave the expressions of $\zeta_{i}$ and $I$ in [Po2] and we will present these results in each type of $\mathfrak{g}$ below.

	Define the power sum of order $2$ by
	\begin{align}
		S(\zeta) = \sum_{i\in I} [(\zeta_{i} + r_{i} + \alpha)^{2} - (r_{i} + \alpha)^{2}],
	\end{align}
	where $\alpha$ and the index set $I$ for every classical Lie algebra are listed in Table 1
	\begin{table}[h]
		\caption{}
		\begin{center}
			\begin{tabular}{| c | c | c | c |}
				\hline
				classical type & \multicolumn{1}{c|}{notation} & \multicolumn{1}{c|}{$\alpha$} & \multicolumn{1}{c|}{$I$} \\ \hline
				$A_{l}$ & $\mathfrak{sl}_{l+1}(\mathbb{C})$ & $\frac{l}{2}$ & $1, 2, \dots, l+1$ \\ \hline
				$B_{l}$ & $\mathfrak{so}_{2l+1}(\mathbb{C})$ & $l - \frac{1}{2}$ & $0, \pm 1, \dots, \pm l$ \\ \hline
				$C_{l}$ & $\mathfrak{sp}_{2l}(\mathbb{C})$ & $l$ & $\pm 1, \pm 2, \dots, \pm l$ \\ \hline
				$D_{l}$ & $\mathfrak{so}_{2l}(\mathbb{C})$ & $l-1$ & $\pm 1, \pm 2, \dots, \pm l$ \\ \hline
			\end{tabular}\\
		\end{center}
	\end{table}
   and
	\begin{equation} \label{r_{i}}
		r_{i} = (\alpha + 1)\varepsilon_{i} - i, \quad \varepsilon_{i} = 
		\begin{cases}
			1 & \text{ if } \quad i > 0 \\
			0 & \text{ if } \quad i = 0 \\
			-1 & \text{ if } \quad i < 0
		\end{cases}.
	\end{equation}
	The components $\zeta_{i}$ and $r_{j}$ have the anti-symmetry property
	\begin{align}
		\zeta_{-i} = -\zeta_{i}, \quad r_{-i} = -r_{i}, \quad \zeta_{0} = r_{0} = 0,
	\end{align}
	by virtue of which only components with $i > 0$ are independent.

	The eigenvalue of second order Casimir operator $C$ on $V_{\zeta}$ is given by the corresponding  infinitesimal character on $C$.
	It follows from [Po2] that 
	\begin{align}\label{Z1}
		\chi_{\zeta}(C) = S(\zeta).
	\end{align}

	For $A_{l}$, write the highest weight $\zeta = a_{1}\omega_{1} + \dots + a_{l}\omega_{l}$. 
	Let $f_{i}(\zeta) =\sum\limits_{k=i}^{l} a_{k},  i = 1, 2, \dots, l$.
	The Young tableau $[f(\zeta)] = [f_{1}(\zeta), \dots, f_{l}(\zeta), f_{l+1}(\zeta)]$ is given by $f_{i}(\zeta)$, which is the number of cells in the row $i$ of $[f(\zeta)]$ and $f_{1}(\zeta) \geq \dots \geq f_{l}(\zeta) \geq f_{l+1}(\zeta) = 0$.
	From [Po1], one knows that the eigenvalue $\chi_{\zeta}(C)$  of the second order Casimir element $C$ acting on irreducible $A_{l}$-module $V_{\zeta}$ is of the form:
	\begin{align} \label{C_{2} of A_{l}}
		\chi_{\zeta}(C) = \sum_{i = 1}^{l+1} [(\zeta_{i} - i + l + 1)^{2} - (l + 1 - i)^{2}],
	\end{align}
	where 
	$$\zeta_{i} = f_{i}(\zeta) - \frac{1}{l+1} \sum_{j = 1}^{l+1}f_{j}(\zeta), \quad i = 1, \dots, l+1.$$
	
	\medskip

	For $C_{l}$, write the highest weight $\zeta = a_{1}\omega_{1} + \dots + a_{l}\omega_{l}$.
	The Young tableau $[f(\zeta)] = [f_{1}(\zeta), f_{2}(\zeta), \dots, f_{l}(\zeta)]$ is given by $f_{i}(\zeta) =\sum\limits_{k=i}^{l} a_{k},  i = 1, 2, \dots, l$.
	From [Po2], the eigenvalue $\chi_{\zeta}(C)$ of the second order Casimir element $C$ acting on irreducible $C_{l}$-module $V_{\zeta}$ is of the form:
	\begin{align} \label{C_{2} of C_{l}}
		\chi_{\zeta}(C) = \sum_{i = 1}^{l} [(\zeta_{i} - i + 2l + 1)^{2} + (\zeta_{i} - i + 1)^{2}] - \sum_{i = 1}^{l} [(r_{i} + \alpha)^{2} + (r_{-i} + \alpha)^{2}],
	\end{align}
	where $\alpha = l$ in Table 1 and 
	$$\zeta_{i} = f_i(\zeta), i = 1, \dots, l.$$
	
	\medskip

	For $B_{l}$, write the highest weight $\zeta = a_{1}\omega_{1} + \dots + a_{l}\omega_{l}$.
	From [Po2], one knows that the $C$-eigenvalue of irreducible $B_{l}$-module $V_{\zeta}$ is of  the form:
	\begin{align}\label{C_{2} of B_{l}}
		\chi_{\zeta}(C) = \alpha^{2} + \sum_{i = 1}^{l} [(\zeta_{i} - i + 2l)^{2} + (\zeta_{i} - i + 1)^{2}] - \sum_{i = -l}^{l} (r_{i} + \alpha)^{2},
	\end{align}
	where $\alpha = l - \frac{1}{2}$ and 		
	\begin{equation} \label{B_{l}}
		\begin{split}
			\zeta_{i} = a_{i} + a_{i+1} + \dots + &a_{l-1} + \frac{1}{2}a_{l}, \quad i = 1,2, \dots, l-1, \\ 
			&\zeta_{l} = \frac{1}{2}a_{l}.
		\end{split}
	\end{equation}
	
	\medskip

	For $D_{l}$, write the highest weight $\zeta = a_{1}\omega_{1} + \dots + a_{l}\omega_{l}$. 
	There are two inequivalent spin $\mathfrak{g}$-modules $U^{+}$ and $U^{-}$.
	If $l$ is even, then $U^{+} = V_{\omega_{l}}$ and $U^{-} = V_{\omega_{l-1}}$. 
	Denote
	\begin{equation} \label{D_{l} even}
		\begin{split}
			\zeta_{i} = a_{i} + &a_{i+1} + \dots + a_{l-2} + \frac{1}{2}(a_{l-1} + a_{l}), \quad i = 1,2, \dots, l-2, \\
			&\zeta_{l-1} = \frac{1}{2}(a_{l-1} + a_{l}), \quad \zeta_{l} = \frac{1}{2}(a_{l} - a_{l-1}).
		\end{split}
	\end{equation}
	If $l$ is odd, then $U^{+} = V_{\omega_{l-1}}$ and $U^{-} = V_{\omega_{l}}$. 
	Denote
	\begin{equation} \label{D_{l} odd}
		\begin{split}
			\zeta_{i} = a_{i} + &a_{i+1} + \dots + a_{l-2} + \frac{1}{2}(a_{l-1} + a_{l}), \quad i = 1,2, \dots, l-2, \\
			&\zeta_{l-1} = \frac{1}{2}(a_{l-1} + a_{l}), \quad \zeta_{l} = \frac{1}{2}(a_{l-1} - a_{l}).
		\end{split}
	\end{equation}
	By [Po2], one knows that the $C$-eigenvalue of irreducible $D_{l}$-module $V_{\zeta}$ is of the form:
	\begin{align} \label{C_{2} of D_{l}}
		\chi_{\zeta}(C) = \sum_{i = 1}^{l} [(\zeta_{i} - i + 2l - 1)^{2} + (\zeta_{i} - i + 1)^{2}] - \sum_{i = 1}^{l} [(r_{i} + \alpha)^{2} + (r_{-i} + \alpha)^{2}],
	\end{align}
	where $\alpha = l-1$ and $\zeta_{i}, i = 1, \dots, l$ satisfy $(A.9)$ for $l$ even and $(A.10)$ for $l$ odd.

	\section{M-type matrices for complex simple Lie algebras}
	
	In this appendix, we give the M-type matrices $M_{\lambda}(C)$ with strongly multiplicity free weights for complex simple Lie algebras.
	
	\subsection{$A_{l}$}

	For $A_{l}$, let $E_{i, j}$ be the $(l+1) \times (l+1)$ matrix whose element on row $i$ and column $j$ is $1$ and others are $0$ and $I_{l+1}$ be the identity matrix.
	Set
	$$X_{i, i} = E_{i, i} - \frac{1}{l+1} I_{l+1}, \ 1 \leq i \leq l+1,$$
	$$X_{i, j} = E_{i, j}, \ 1 \leq i \ne j \leq l+1.$$
	It is obvious that $\{X_{i, j}, X_{k, k}| 1 \leq i \ne j \leq l+1, 1 \leq k \leq l \}$ is a set of the basis for $A_{l}$.
	Based on the Killing form $B = \tr_{\omega_1}$ for $A_{l}$, we have the dual elements 
	$$X_{i, i}^{*} = X_{i,i}-X_{l+1,l+1}, \ 1 \leq i \leq l,$$
	$$X_{i ,j}^{*} = X_{j, i}, \ 1 \leq i \ne j \leq l+1.$$
	
	For the natural module $V_{\omega_1}$, one knows that
	$$\pi_{\omega_1}(X_{i,i})=E_{i,i}-\frac{1}{l+1} I_{l+1}, \ 1 \leq i  \leq l,$$
	$$\pi_{\omega_1}(X_{i,j})=E_{i,j}, \ 1 \leq i \neq j \leq l+1.$$
	Then we obtain
	\begin{equation} \label{M_{omega_1}(C) for A_{l}}
		M_{\omega_1}(C) = \sum\limits_{1 \leq i \ne j \leq l+1} \pi_{\omega_1}(X_{i, j}) \otimes X_{i, j}^{*} + \sum_{1 \leq i \leq l} \pi_{\omega_1}(X_{i, i}) \otimes X_{i, i}^{*}
	\end{equation}
	$$=\sum\limits_{1 \leq i\neq j \leq l+1} E_{i,j}\otimes X_{j,i} + \sum\limits_{1 \leq i \leq l}(E_{i,i}-\frac{1}{l+1}I_{l+1}) \otimes (X_{i,i}-X_{l+1,l+1})$$
	$$=\sum\limits_{1 \leq i, j \leq l+1}E_{i,j}\otimes X_{j,i}.$$
	
	Since $V_{\omega_{l}}$ is the dual module of $V_{\omega_{1}}$ for $A_{l}$, the corresponding M-type matrix $M_{\omega_{l}}(C)$ is also of the form \eqref{M_{omega_1}(C) for A_{l}}.
	
	For irreducible $A_{1}$-module $V_{k \omega_{1}} (k \in \mathbb{N})$, M-type matrix $M_{k \omega_{1}}(C)$ is obtained by Kirillov [Ki1].

	\subsection{$B_{l}$}
	
	For $B_{l}$, let $E_{i, j}$ be the $(2l+1) \times (2l+1)$ matrix whose element on row $i$ and column $j$ is $1$ and others are $0$.
	Set
	$$X_{i, j} = E_{i, j} - E_{j, i}, \ 1 \leq i \ne j \leq 2l+1.$$
	It is obvious that $\{X_{i, j} | 1 \leq i < j \leq 2l+1 \}$ is a set of the basis for $B_{l}$.
	Based on the Killing form $B = \tr_{\omega_1}$ for $B_{l}$, we have the dual elements
	$$ X_{i,j}^{*}=-\frac{1}{2}X_{i,j}, \ 1 \leq i \neq j  \leq 2l+1. $$
	
	For the natural module $V_{\omega_1}$, one knows that 
	$$ \pi_{\omega_1}(X_{i,j})=E_{i,j}-E_{j,i}$$
	for $1 \leq i \ne j \leq 2l+1.$
	Then we obtain
	\begin{equation}
		M_{\omega_1}(C) = \sum\limits_{1 \leq i < j \leq 2l+1}\pi_{\omega_1}(X_{i, j})\otimes X_{i, j}^{*} = -\frac{1}{2}\sum\limits_{1 \leq i< j \leq 2l+1}(E_{i,j}-E_{j,i}) \otimes X_{i,j}.
	\end{equation}

	\subsection{$C_{l}$}
	
	For $C_{l}$, let the indices $i,j $ run from $-l$ to $l$ excluding zero and $E_{i, j}$ be the $2l \times 2l$ matrix whose element on row $k$ and column $m$ is $1$ and others are $0$, where positive integers $k \equiv i (\text{mod } l)$ and $m \equiv j (\text{mod } l)$.
	Define $\epsilon_i = 1$ for $1 \leq i \leq l$  and $\epsilon_i = -1$ for $-l \leq i \leq -1$ and set
	$$X_{i, j} = E_{i,j}-\epsilon_i \epsilon_j E_{-j,-i}$$
	for $-l \leq i, j \leq l.$
	They satisfy the symmetry:
	$$X_{i, j}=-\epsilon_i \epsilon_j X_{-j,-i}.$$
	It is obvious that $\{X_{k, m}, X_{i, -j}, X_{-i, j} | 1 \leq k, m \leq l, 1 \leq i \leq j \leq l \}$ is a set of the basis for $C_{l}$.
	Based on the Killing form $B = \tr_{\omega_1}$ for $C_{l}$, we have the dual elements 
	$$X_{k, m}^{*} = \frac{1}{2} X_{m, k}, \ 1 \leq k, m \leq l,$$
	$$X_{i, -j}^{*} = \frac{1}{2} X_{-i, j}, \ 1 \leq i \ne j \leq l,$$
	$$X_{i, -i}^{*} = \frac{1}{4} X_{-i, i}, \ 1 \leq i \leq l.$$
	
	For the natural module $V_{\omega_1}$, one knows that  
	$$  \pi_{\omega_{1}}(X_{i,j})=E_{i,j}-\epsilon_i \epsilon_j E_{-j,-i}$$
	for $-l \leq i, j \leq l$.
	Then we obtain
	\begin{equation}
		M_{\omega_1}(C) = \sum\limits_{1 \leq i, j \leq l} \pi_{\omega_1}(X_{i, j}) \otimes X_{i, j}^{*} 
	\end{equation}
	$$+ \sum_{1 \leq i \leq j \leq l} [\pi_{\omega_1}(X_{i, -j}) \otimes X_{i, -j}^{*} + \pi_{\omega_1}(X_{-i, j}) \otimes X_{-i, j}^{*} ]$$
	$$
	= \frac{1}{2}[\sum\limits_{l \leq i,j \leq l}(E_{i,j}-E_{-j,-i})\otimes X_{j,i}
	+ \sum\limits_{ 1 \leq i < j \leq l}(E_{i,-j}+E_{j,-i})\otimes X_{-i,j}
	$$
	$$
	+ \sum\limits_{1 \leq i < j \leq l}(E_{-i,j}+E_{-j,i})\otimes X_{i,-j}
	+ \sum\limits_{ 1 \leq i \leq l}E_{i,-i}\otimes X_{-i,i}
	+ \sum\limits_{ 1 \leq i \leq l}E_{-i,i}\otimes X_{i,-i}].
	$$
	It can be written in a block form:
	\begin{equation}
		M_{\omega_1}(C) = 
		\begin{pmatrix}
			A & B      \\
			D & -A^{T}
		\end{pmatrix},
	\end{equation}
	where
	$$ A = (-\frac{1}{2}X_{j, i})_{1 \leq i, j \leq l}, \ B = (\frac{1}{2}X_{-i,j})_{1 \leq i, j \leq l} = B^{T}, \ D = (\frac{1}{2}X_{i,-j})_{1 \leq i, j \leq l} = D^{T}.$$

	\subsection{$G_{2}$}
	
	For exceptional Lie algebra $G_{2}$, let $E_{i, j}$ be the $7 \times 7$ matrix whose element on row $i$ and column $j$ is $1$ and others are $0$ and matrix $A^{T}$ be the transpose of matrix $A$.
	Since exceptional Lie algebra $G_{2}$ is realized by the subalgebra of classical simple Lie algebra $B_{3}$, $\phi = \{ \pm\lambda_{i}, \lambda_{i} - \lambda_{j} | 1 \leq i \ne j \leq 3 \}$ is the root system of $G_{2}$, where $\lambda_{i}$ satisfies $\lambda_{i}(\text{diag}(0, x_{1}, x_{2}, x_{3}, -x_{1}, $ $-x_{2}, -x_{3})) = x_{i}$ for $x_{i} \in \mathbb{C}$.
	Then we have the Cartan decomposition 
	$$G_{2} = \mathfrak{h}_{0} + \sum\limits_{\alpha \in \Phi} \mathbb{C}G_{\alpha},$$
	where
	$$\mathfrak{h}_{0} = \{ \text{diag}(0, x_{1}, x_{2}, x_{3}, -x_{1}, -x_{2}, -x_{3}) | \sum\limits_{ i = 1}^{3} x_{i} = 0, x_{i} \in \mathbb{C} \},$$
	$$G_{\lambda_{1} - \lambda_{2}} = G_{\lambda_{2} - \lambda_{1}}^{T} = (E_{23} - E_{65})/ \sqrt{2},$$
	$$G_{\lambda_{1} - \lambda_{3}} = G_{\lambda_{3} - \lambda_{1}}^{T} = (E_{24} - E_{75})/ \sqrt{2},$$
	$$G_{\lambda_{2} - \lambda_{3}} = G_{\lambda_{3} - \lambda_{2}}^{T} = (E_{34} - E_{76})/ \sqrt{2},$$
	$$G_{\lambda_{1}} = -G_{-\lambda_{1}}^{T} = (\sqrt{2}(E_{12} - E_{51}) - (E_{37} - E_{46}))/ \sqrt{-6},$$
	$$G_{\lambda_{2}} = -G_{-\lambda_{2}}^{T} = (\sqrt{2}(E_{13} - E_{61}) - (E_{27} - E_{45}))/ \sqrt{-6},$$
	$$G_{\lambda_{3}} = -G_{-\lambda_{3}}^{T} = (\sqrt{2}(E_{14} - E_{71}) - (E_{26} - E_{35}))/ \sqrt{-6}.$$
	Let 
	$$ h_{1} = \text{diag}(0, 0, \frac{1}{2}, -\frac{1}{2}, 0, -\frac{1}{2}, \frac{1}{2}), $$
	$$ h_{2} = \text{diag}(0, \frac{1}{2}, 0,  -\frac{1}{2}, -\frac{1}{2}, 0, \frac{1}{2}).$$
	Based on the Killing form $B = \tr_{\omega_1}$ for $G_{2}$, we have the dual elements
	$$h_{i}^{*} = h_{i}, \quad i = 1,2,$$
	$$G_{\alpha}^{*} = G_{-\alpha}, \quad \alpha \in \phi.$$

	For the natural module $V_{\omega_1}$, one knows that
	$$\pi_{\omega_1}(h_{i}) = h_{i}, \quad i = 1, 2,$$
	$$\pi_{\omega_1}(G_{\alpha}) = G_{\alpha}, \quad \alpha \in \phi.$$
	Then we have
	\begin{equation}
		M_{\omega_1}(C)=\pi_{\omega_1}(h_{1}) \otimes h_{1}^{*} + \pi_{\omega_1}(h_{2}) \otimes h_{2}^{*} + \sum\limits_{\alpha \in \phi} \pi_{\omega_1}(G_{\alpha})\otimes G_{\alpha}^{*}
	\end{equation}
	$$= h_{1} \otimes h_{1} + h_{2} \otimes h_{2} + \sum\limits_{ 1 \leq i < j \leq 3 } (G_{\lambda_{i} - \lambda_{j}} \otimes G_{\lambda_{j} - \lambda_{i}} + G_{\lambda_{j} - \lambda_{i}} \otimes G_{\lambda_{i} - \lambda_{j}}) $$
	$$ + \sum\limits_{ 1 \leq i \leq 3 }( G_{\lambda_{i}} \otimes G_{\lambda_{-i}} +  G_{\lambda_{-i}} \otimes G_{\lambda_{i}}).$$

\end{document}